\documentclass[12pt]{amsart}
\usepackage[letterpaper,top=2.5cm,bottom=2.5cm,left=3cm,right=3cm]{geometry}

\usepackage[backend=biber,style=numeric,sorting=nyt,doi=false,url=false,isbn=false,giveninits=true]{biblatex}
\usepackage{hyperref}
\usepackage{tabularx,booktabs}
\usepackage{caption}
\usepackage{amsmath} 
\usepackage{bbm}
\usepackage{amsfonts}
\usepackage{amscd}
\usepackage{amsthm}
\usepackage{amssymb} 
\usepackage{cancel}
\usepackage{latexsym}
\usepackage{eufrak}
\usepackage{euscript}
\usepackage{epsfig}
\usepackage{graphics}
\usepackage{array}
\usepackage{enumerate}
\usepackage{dsfont}
\usepackage{color}
\usepackage{wasysym}
\usepackage{multirow}
\usepackage{pdfsync}

\newcommand{\bel}[1]{\begin{equation}\label{#1}}

\newcommand{\be}{\begin{equation}}

\newcommand{\qe}{\end{equation}}

\newcommand{\R}{{\mathbb R}}
\newcommand{\N}{{\mathbb N}}
\newcommand{\Z}{{\mathbb Z}}
\newcommand{\C}{{\mathbb C}}

\newcommand{\hh}{h_{\mathbb{S}}}

\newcommand{\CC}{\mathcal{L}_{\alpha,d}}
\newcommand{\Sl}{\mathbb{S}_{\alpha,l}^{d-1}}

\newcommand{\sss}{\mathfrak{s}}
\newcommand{\bb}{\mathfrak{b}}

\newcommand{\PP}{\widetilde{\mathbf{P}}_4}

\newtheorem{theorem}{Theorem}[section]

\newtheorem{lemma}[theorem]{Lemma}

\newtheorem{definition}[theorem]{Definition}
\newtheorem{remark}[theorem]{Remark}

\newtheorem{prop}[theorem]{Proposition}

\begin{document}

\title[Sharp dispersive estimates for the wave equation]{Sharp dispersive estimates for the wave equation on the 5-dimensional lattice graph}

\author{Cheng Bi}
\address{Cheng Bi: School of Mathematical Sciences, Fudan University, Shanghai 200433, China}
\email{\href{mailto:cbi21@m.fudan.edu.cn}{cbi21@m.fudan.edu.cn}}

\author{Jiawei Cheng}
\address{Jiawei Cheng: School of Mathematical Sciences, Fudan University, Shanghai 200433, China}
\email{\href{mailto:chengjw21@m.fudan.edu.cn}{chengjw21@m.fudan.edu.cn}}

\author{Bobo Hua}
\address{Bobo Hua: School of Mathematical Sciences, LMNS, Fudan University, Shanghai 200433, China; Shanghai Center for Mathematical Sciences, Fudan University, Shanghai 200433, China}
\email{\href{mailto:bobohua@fudan.edu.cn}{bobohua@fudan.edu.cn}}

\begin{abstract}
Schultz \cite{S98} proved  dispersive estimates for the wave equation on lattice graphs $\Z^d$ for $d=2,3,$ which was extended to $d=4$ in \cite{BCH23}. By Newton polyhedra and the algorithm introduced by Karpushkin \cite{K83}, we further extend the result to $d=5:$ the sharp decay rate of the fundamental solution of the wave equation on $\mathbb{Z}^5$ is $|t|^{-\frac{11}{6}}.$ Moreover, we prove Strichartz estimates and give applications to nonlinear equations.
\end{abstract}

\maketitle
\numberwithin{equation}{section}

\section{introduction}\label{sec-intro}

Discrete dispersive equations in the form of difference equations have 
attracted much attention in the literature of mathematics and physics, 
since they constitute a natural way to approach numerically real physical laws. 
Indeed, spatial discretization would be the first step to implement finite difference schemes, transfering an equation on a continuum domain to that on a lattice graph. As for discrete wave equations, they 
appear 
in physical applications such as lattice dynamics and can be used to describe the vibrations of atoms inside crystals. A fundamental model is the monotonic chains, 
see \cite{D93,FH10,K11}. On general graphs, wave equations have been studied in \cite{FT04,HH20,LX22}.

In this paper, we consider dispersive and Strichartz estimates for the discrete wave equation
\begin{equation}\label{equ-wave equ}
    \left\{
    \begin{aligned}
        & \partial_{t}^2 u(x,t) - \Delta u(x,t) = F(x,t), \\
        & u(x,0) = f_1(x),\quad\partial_t u(x,0) = f_2(x),\quad (x,t)\in\Z^d\times \R.
    \end{aligned}
    \right.
\end{equation}
Here the discrete Laplacian $\Delta$ is defined by
\begin{equation*}
    \Delta u (x,t) := \sum_{j=1}^d \left(u(x+\boldsymbol{e_j},t)+u(x-\boldsymbol{e_j},t)-2 u(x,t)\right), 
\end{equation*}
where $\{\boldsymbol{e_j}\}_{j=1}^d$ is the standard basis of the lattice $\Z^d$.

By the discrete Fourier transform, the fundamental solution of \eqref{equ-wave equ} is given by
\begin{equation}\label{equ-fundamental G}
    G(x,t) = \frac{1}{(2\pi)^d}\int_{\mathbb{T}^d}  e^{ix \cdot \xi} \; \frac{\sin (t \, \omega (\xi))}{\omega (\xi)}\,d\xi,\quad \mbox{with} \ \ \omega(\xi) = \left({\sum_{j=1}^{d} (2-2\cos \xi_j)}\right)^{\frac{1}{2}},
\end{equation}
where $\mathbb{T}^d = [-\pi,\pi]^d$ and $x\cdot\xi= \sum_{j=1}^d x_j \xi_j$, see Section \ref{ssec-basics}.

The pioneering work on sharp dispersive estimates for $G$ 
was initiated in Schultz \cite{S98}, where he proves that $G$ decays like $|t|^{-2/3}$ and $|t|^{-7/6}$ when $d=2$ and $3$ respectively. On $\Z^4$, the authors \cite[Theorem 1.1]{BCH23} proved a sharp upper bound  of order $|t|^{-3/2}\log|t|$ (or $\mathcal{O}(|t|^{-3/2}\log |t|)$, in short) as $t\rightarrow\infty$.  In all cases $d=2,3,4$, the following oscillatory integral plays an important role,
\begin{equation}\label{equ-fundamental I}
    I(v,t) := \frac{1}{(2\pi)^d}\int_{\mathbb{T}^d} e^{i t \phi(v,\xi)}\,\frac{1}{\omega(\xi)}\, d\xi,\quad \phi(v,\xi):=v \cdot \xi - \omega(\xi).
\end{equation}
Note that $G(x,t)$ is the imaginary part of $-I(x/t,t)$.
 
Based on the analysis in \cite{S96,S98},  we know the main obstacle is to describe long-time asymptotic behaviour for $I$ when $|v|$ is small, cf. Section \ref{sec-main thm}. In this case, $\phi(v,\cdot)$ has degenerate critical points and the method of stationary phase breaks down. When $d\le 3$, in the terminology of \cite{AGV12}, only stable singularities $A_k$ ($k\le 5$) and $D_4$ appear. As $d$ increases, however, the singularity type becomes complicated.
Instead of classifying all kinds of singularities, we seek a suitable way to obtain the stability of $I$ and find its optimal decay rate, uniformly in $v$. Here uniformity in $v$ for the decay is the key issue.

In the sense of V. I. Arnold \cite{Ar73}, for an oscillatory integral
\begin{equation}\label{equ-J}
    J(t,S,\psi) \,:=\, \int_{\R^d}\, e^{i t S(\xi)}\, \psi(\xi)\, d\xi,
\end{equation} 
establishing the uniform estimate is to determine whether the decay estimate of $J(t,S,\psi)$ could be extended to $J(t,S+P,\psi)$ for 
$P$ with sufficiently small norm. If $d=1$, this is Van der Corput lemma, cf. \cite[Chapter 8]{S93}. If $d=2$, the answer is also affirmative by \cite{K84}. However, Arnold's conjecture  is not always true when $d\ge 3$, even in the case that $P$ is linear. See the counterexamples in  \cite{G07,V76}. For more results we refer to \cite{G14,IM11tran,PSS99}, all these work are closely related to Newton polyhedra.

In general, when $d\ge 3$ and $S$ has degenerate critical points, it is difficult to establish the sharp uniform estimate of $J(t,S+P,\psi)$. Even for $P=0$, it is still complicated to determine the oscillation index (cf. \eqref{equ-expan} below) of $S$.
Nevertheless, Karpushkin proposed an algorithm in \cite{K83} to determine the uniform upper bound of $J(t,S+P,\psi)$ when $S$ is quasi-homogeneous.  This algorithm reduces $S+P$ to a 
family of polynomial phases 
and can be iterated several times to get the results, see Section \ref{sec-uniform esti}.

For \emph{odd} $d \geq 5$, we proved in the previous paper \cite[Theorem 1.5]{BCH23} the upper bound $\mathcal{O}(|t|^{-\frac{2d+1}{6}})$ for $I(v,t)$ with a fixed $v=v(d):=(\frac{1}{\sqrt{2d}},\cdots,\frac{1}{\sqrt{2d}}) \in \R^d$. This velocity governs the decay rates when $d=3,4$. Due to the lack of uniformity in $v$, the dispersive estimate for $G$
remains open when $d \geq 5$.
In this paper, we prove the sharp decay estimate using Newton polyhedra and the algorithm of Karpushkin for $d=5$.

\begin{theorem}\label{thm-main thm}
    There exists $C>0$ independent of $x\in\Z^5$ such that 
    \begin{equation*}
        |G(x,t)| \leqslant C(1+|t|)^{-\frac{11}{6}}.
    \end{equation*}
\end{theorem} 

\begin{remark}
  (a)
    A related model is the discrete Klein-Gordon equation
    \begin{equation*}
        \partial_t^2 u(x,t) - \Delta u(x,t) +m_*^2 \, u(x,t)= 0 ,
    \end{equation*}
    where $m_*>0$ is the mass parameter. This model has been studied in Cuenin and Ikromov \cite{CI21} for $d=2,3,4$ (see also \cite{BG17}). Theorem \ref{thm-main thm} gives a partial answer to the conjecture in  \cite{CI21}. Indeed, let $G_*$ be the corresponding fundamental solution, then $G$ has additional singularity at the origin compared with $G_*$. When $d\ge 3$, by the techniques in \cite{S96,S98} and slight modifications of our proof, one can show that $G_*$ shares the same decay estimate with $G$. This yields sharp dispersive estimates for the discrete Klein-Gordon equation on $\Z^d$.

    (b)Theorem \ref{thm-main thm} is sharp in the sense that there exists $c_0\ne 0$ such that
    \begin{equation}\label{equ-sharp}
          \lim_{t\rightarrow \infty}  \left|t^{\frac{11}{6}} \, I(v(5),t)\right| = c_0,
    \end{equation}
    where $v(5)=(\frac{1}{\sqrt{10}},\cdots,\frac{1}{\sqrt{10}}) \in\R^5$, see \hyperref[Appendix]{Appendix}.
\end{remark}

For all $d\ge 3$, the main ingredient of the proof of the theorem is the uniform estimate of \eqref{equ-J} with phase of the type
\begin{equation}\label{equ-PP}
    \mathbf{P}_{m}(z)\equiv\mathbf{P}_{m,d}^{\Psi} (z):=\left(\sum_{j=1}^{m}z_j\right)^3-\sum_{j=1}^{m} z_j^3 \,+\,\Psi_{m,d}(z_{m+1},\cdots,z_d), \ \ 2\le m\le d-1,
\end{equation}
where $\Psi_{m,d}$ is a nondegenerate quadratic form. This polynomial appears naturally in the study of \eqref{equ-fundamental I}. 
 Indeed, if $\phi(v,\cdot)$ has a degenerate critical point $\xi_0$ with corank $ m \in [2\,,d-1]$ (i.e. rankHess$_\xi\phi(v,\xi_0)=d-m$), 
then $\phi(v,\cdot)$ can be expressed as
$ \mathbf{P}_m(z) + \mathcal{O}\left(|z|^4\right)$ near $\xi_0$. 
In particular, $I(v(d),t)$ corresponds to the crucial phase $\mathbf{P}_{d-1}$, see \cite[Lemma 3.6]{BCH23}.
Note also that $\mathbf{P}_2$ has $D_4$ type singularity, and $\mathbf{P}_3$ can be reduced to $z_1 z_2 z_3 + \Psi_{3,d}$ (in this case $\phi$ has $T_{4,4,4}$ type singularity, cf. \cite{CI21}).

In our context $d=5$, the new case $\mathbf{P}_4$ is the most complicated one, which has singularity of class $O$, cf. \cite[p. 253]{AGV12}. 
Karpushkin's algorithm will be applied to this phase, see Section \ref{sec-uniform esti}. More precisely, we consider the deformation of $\mathbf{P}_4$ with rank $\ge 2$. 
Using change of variables, we reduce $\mathbf{P}_4$ to homogenous polynomials with three variables, involving many parameters. 
Then we repeat this process and reduce the problem to oscillatory integrals in $\R^2$, where the phases are expressed in adapted coordinate systems (cf. Section \ref{ssec-newton}). Finally, we use results in \cite{K84,V76} to obtain the desired bound. 

Note that the key difficulty is to deal with highly degenerate oscillatory integrals, and our approach is different from that in \cite{CI21} for the discrete Klein-Gordon equation. To the best of the authors' knowledge, it is the first time to adopt Karpushkin's algorithm for estimating such oscillatory integrals with phase of corank $>3$. Our results are indeed valid for any analytic perturbation. However, our iteration approach will be tedious as $d$ increases. Furthermore, as a germ at the origin, $\mathbf{P}_{d-1}$ is not finitely determined (cf. \cite[Chapter 5]{M21}) for \emph{even} $d\ge 4$ since $0$ is not its isolated critical point. For these reasons, the dispersive estimates when $d\ge 6$ remain open.

By Theorem \ref{thm-main thm} and a well-known result in \cite{KT98}, we obtain the following Strichartz estimate. Also, a standard argument can be used for the global existence of the solutions to nonlinear equation \eqref{equ-wave equ} with small initial data, see Theorem \ref{thm-global-1}. 
\begin{theorem}\label{thm-stri}
    Let $d=5$ and $u$ be the solution to \eqref{equ-wave equ}. If indices $q,r,\widetilde{q},\widetilde{r}$ satisfy
    \begin{equation}\label{equ-stri index}
        q,r,\tilde{q},\tilde{r} \geqslant 2,\quad \frac{1}{q} \leq \frac{11}{6}\left(\frac{1}{2}-\frac{1}{r}\right)\quad\mbox{and}\quad
        \frac{1}{\tilde{q}} \leq \frac{11}{6}\left(\frac{1}{2}-\frac{1}{\tilde{r}}\right),
    \end{equation}
    then there exists $C=C(q,r,\widetilde{q},\widetilde{r})$ such that
    \begin{equation*}
    \|u\|_{L^q_t \ell^r} \leqslant C\left(\|f_1\|_{\ell^2} + \|f_2\|_{\ell^{\frac{10}{7}}} + \|F\|_{L_t^{\tilde{q}'} \ell^{\frac{5\tilde{r}'}{5+\tilde{r}'}}}\right),
    \end{equation*}
    where $p'$ denotes the conjugate index of $p$ for any $p\in[1,\infty]$.
\end{theorem}

The paper is organized as follows.
We recall basic facts about the discrete setting and Newton polyhedra in Section \ref{sec-preli}. We also state some estimates concerning the stability of oscillatory integral in this section. In Section \ref{sec-main thm}, we give the proof of Theorem \ref{thm-main thm}. In Section \ref{sec-uniform esti}, we prove the key result, Proposition \ref{prop-Q_11 3+4}, which is crucial for the proof of Theorem \ref{thm-main thm}. In Section \ref{sec-nonlinear}, we prove Strichartz estimates and give applications to the nonlinear equations. In \hyperref[Appendix]{Appendix} we show the sharpness of Theorem \ref{thm-main thm}.\\

\noindent
\textbf{Notation.}
We  use  $|\cdot|$ and $\cdot$ to denote the length and the inner product on Euclidean spaces, respectively, and $\mathbb{A}^T$ the transpose of matrix $\mathbb{A}$. Let $B_{\R^d}(\xi,r)$ (resp. $B_{\C^d}(\xi,r)$) be the usual open ball in $\R^d$ (resp. $\C^d$) with center $\xi$ and radius $r$, while $\overline{B}_{\R^d}(\xi,r)$ (resp. $\overline{B}_{\C^d}(\xi,r)$) denotes its closure.

The symbols $C, c$ will be used throughout to denote implicit positive constants independent of 
$(x,t)$, which may
vary from one line to the next. For non-negative functions $f$ and $g$, we adopt the notation $f \lesssim g$ if there exists  $C>0$ such that $f \le C g$. For any $\xi\in\R^d$, the translation $\boldsymbol{\tau}_\xi f(z):=f(z+\xi)$.
Also, we use $\partial_j$ to denote $\frac{\partial}{\partial \xi_j}$, $\nabla$ (or $\nabla_\xi$) the usual gradient, and Hess (or Hess$_\xi$)
 the Hessian matrix.
Moreover, we use $ \boldsymbol{w}(U)$ (resp. $\boldsymbol{w}^{-1}(U)$) to denote the image (resp. preimage) of $U$ under map $\boldsymbol{w}$.

\section{Preliminaries}\label{sec-preli}
\subsection{The discrete setting}\label{ssec-basics}

We denote by $\Z^d$ the standard $d$-dimensional integer latticc graph in $\R^d$, that is, $\Z^d := \{x=(x_1,\cdots,x_d)\in \R^d: x_j \in \Z,\, j=1,2,\cdots,d\}$.
For $p \in [1,\infty]$, $\ell^p(\Z^d)$ is the $\ell^p$-space of functions on $\Z^d$ with respect to the counting measure, which is a Banach space endowed with the norm 
\begin{equation*}
    ||f||_{\ell^p} :=\left\{
    \begin{aligned}
        &\left(\sum_{x \in \Z^d} |f(x)|^p\right)^{\frac{1}{p}},\ p\in[1,\infty),\\
        &\sup_{x\in\Z^d} |f(x)|,\ p=\infty.
    \end{aligned}\right.
\end{equation*}
We shall also use $|f|_p$ to denote the $\ell^p$ norm of $f$ for notational convenience. Note that $\ell^p$ spaces are nested, that is, $ \ell^p \subset \ell^q$ for $1 \leqslant p \leqslant q \leqslant \infty$. Moreover, for functions $f,g$ on $\mathbb{Z}^d$,  the convolution product is given by
\begin{equation*}
    f*g(x) := \sum_{y \in \mathbb{Z}^d} f(x-y)g(y),\quad x\in\Z^d.
\end{equation*}

The discrete Fourier transform of function $f$ is given by 
\begin{equation*}
   \Hat{f}(\xi) = \sum_{x\in \mathbb{Z}^d} e^{-i\xi \cdot x}f(x),\quad\xi \in \mathbb{T}^d,
\end{equation*} 
while the inverse transform is defined as
\begin{equation*}
    \Check{f}(x) = \frac{1}{(2\pi)^d} \int_{\mathbb{T}^d} e^{i \xi \cdot x} f(\xi) \,d\xi,\quad x\in \mathbb{Z}^d.
\end{equation*}

Applying the Fourier transform to both sides of (\ref{equ-wave equ}), we get 
\begin{equation*}
\left\{
    \begin{aligned}
        & \partial_t^2 \Hat{u}(\xi,t)   + \omega(\xi)^2\, \Hat{u}(\xi,t)=0,\\
        & \Hat{u}(\xi,0) = \hat{f_1}(\xi),~~\partial _t \Hat{u}(\xi,0) = \Hat{f_2}(\xi),\quad\xi\in\mathbb{T}^d,
    \end{aligned}
\right.
\end{equation*}
which gives
\begin{equation*}
    u(x,t)=\frac{1}{(2\pi)^d} \int_{\mathbb{T}^d} e^{i \xi \cdot x} \left(\cos(t\,\omega(\xi))\hat{f_1}(\xi)+\frac{\sin(t\,\omega(\xi))}{\omega(\xi)} \hat{f_2}(\xi)\right) d\xi,\quad (x,t)\in\Z^d\times\R.
\end{equation*}

In the notion of operator theory,
\begin{equation}\label{equ-solu-semigroup}
    u(\cdot,t)=\cos(t\sqrt{-\Delta})f_1+\frac{\sin (t\sqrt{-\Delta})}{\sqrt{-\Delta}}f_2.
\end{equation}

Without loss of generality, we assume that $f_1 \equiv 0$ unless otherwise stated. 
Then we get $u=f_2*G$, where the $G$ is as in (\ref{equ-fundamental G}). Moreover, let $x=vt$, the relation $\omega(\xi)=\omega(-\xi)$ yields $I(v,t)=I(-v,t)$, which gives
$ G(x,t)=-\mbox{Im}\, I(v,t). $

\subsection{Newton polyhedra}\label{ssec-newton}
Let $S$ be a smooth real-valued function on $\R^d$ and real-analytic at 0 such that 
\begin{equation}\label{equ-SSS}
    S(0)=\nabla S(0)=0.
\end{equation}

 Let $J(t,S,\psi)$ be as in \eqref{equ-J}, 
where $\psi\in C_0^\infty(\R^d)$ with support near the origin. Then the following asymptotic expansion holds (cf. e.g. \cite[p. 181]{AGV12}),
\begin{equation}\label{equ-expan}
    J(t,S,\psi) \approx \sum_{\tau} \sum_{\rho=0}^{d-1} c_{\tau,\rho,\psi} \, t^\tau \log^\rho t,\quad \mbox{as} \ \ t\rightarrow+\infty,
\end{equation}
where $\tau$ runs through finitely many arithmetic progressions not depending on $\psi$, which consists of negative rational numbers.
Let $(\tau_S,\rho_S)$ be the maximum over all pairs $(\tau,\rho)$ in \eqref{equ-expan} under the
lexicographic ordering such that for any neighborhood $U$ of the origin, there exists $\psi\in C_0^{\infty}(U)$  for which $c_{\tau_S,\rho_S,\psi}\ne 0$. We call $\tau_S$ the \emph{oscillation index} of $S$ at 0 and $\rho_S$ its \emph{multiplicity}.

The pioneer work of Varchenko \cite{V76} connects \eqref{equ-expan} with the geometry of Newton polyhedra, which we shall recall in the following part. We will use basic notions from \cite{V76}, see also \cite{AGV12,G09,IM11}.

 The associated Taylor series of $S$ at 0 can be written as
\begin{equation}\label{equ-series}
    S(\xi) = \sum_{\gamma\in \mathcal{T}(S)} s_{\gamma} \, \xi^{\gamma},\quad\ \mbox{where}\ \ \mathcal{T}(S) = \{\gamma \in \mathbb{N}^d :s_{\gamma} \neq 0 \}.
\end{equation}
Without loss of generality, we assume that $\mathcal{T}(S)\ne \emptyset$. The Newton polyhedron of $S$, denoted by $\mathcal{N}(S)$, is the convex hull of the set
\begin{equation*}
    \bigcup_{\gamma \, \in\, \mathcal{T}(S)} \left(\gamma + \mathbb{R}^d_{+}\right), \quad\mbox{where}\ \ \R^d_+=\{(\xi_1,\cdots,\xi_d)\in \R^d: \xi_j\geq 0, \, j=1,\cdots,d \}.
\end{equation*}
 The Newton distance $d_S$ is defined as
\begin{equation*}
        d_S = \inf \{\varrho>0:(\varrho,\varrho,\cdots,\varrho) \in \mathcal{N}(S)\}.
\end{equation*}

The principal face $\mathcal{P}_S$  is the face on $\mathcal{N}(S)$ of minimal dimension containing 
$(d_S,\cdots,d_S)$. In particular, under certain nondegeneracy condition, it is proved in \cite{V76} that the oscillation index of $S$ at $0$ is $\frac{1}{d_S}$ if $d_S>1$.

Since $d_S$ depends on the choice of coordinate systems, the height of $S$ is given by
\begin{equation}\label{equ-height}
    h_S:=\sup \{d_{S,\xi}\},
\end{equation}
where the supremum is taken over all local analytic coordinate systems $\xi$ which preserve the origin, and $d_{S,\xi}$ is the Newton distance in coordinates $\xi$. A given coordinate system $\widetilde{\xi}$ is said to be {adapted} to $S$ if $d_{S,\widetilde{\xi}} = h_S$. 

If $d=2$, the following results, derived by \cite[Proposition 0.7, 0.8]{V76} and \cite[Lemma 7.0]{G09}, can recognize whether a given coordinate system is adapted. Note also that adapted coordinates may not exist when $d\ge 3$ by the counterexample in \cite{V76}.
\begin{prop}\label{prop-adapt}
Let $d=2$, $S$ be as in \eqref{equ-SSS} and one of the following conditions holds:
\begin{itemize}
    \item[(a)] $\mathrm{dim}_{\R^2}(\mathcal{P}_S)=0$, i.e. $\mathcal{P}_S$ is a single point.
    \item[(b)] $\mathcal{P}_S$ is unbounded.
    \item[(c)] $\mathcal{P}_S$ is a compact edge. Moreover,
    \[
    \mathcal{P}_S\subset \{\xi: a_1 \xi_1 + \xi_2 = a_2\}\quad \mbox{with}\ \ a_1,a_2\in\N,
    \]
    and $f_{\mathcal{P}_S}(\cdot\,,1)$ does not have a real root of multiplicity larger than $\frac{a_2}{1+a_1}$, where $f_{\mathcal{P}_S}(\xi)= \sum_{\gamma\in\mathcal{P}_S}s_\gamma \xi^{\gamma}$.
\end{itemize}
Then the coordinate system is adapted.
\end{prop}

For instance, one can verify both $ \xi_1^3 \xi_2$ and $ \xi_1^2 \xi_2^2$ are expressed in adapted coordinates with Newton distance $2$, while for $\mathbf{P}_{d-1}$ the Newton distance is $\frac{6}{2d+1}$, cf. \cite{BCH23}. 

\subsection{Results on uniform estimates}\label{ssec-uniform}
As we mentioned in Section \ref{sec-intro}, it is natural to consider the stability of \eqref{equ-J}. We need some notation initiated from \cite{K83}. 

\begin{definition}
     For any $r,s>0$, the space $\mathcal{H}_r(s)$ is defined as
\[
\mathcal{H}_r(s)=\left\{ P: \
\begin{aligned}
    &P \ \mbox{is holomorphic on}\ B_{\C^d}(0,r), \mbox{continuous on}\\ &\  \overline{B}_{\C^d}(0,r),\ \mbox{and}\ |P({w})|<s,\ \forall \, {w}\in \overline{B}_{\C^d}(0,r)
\end{aligned}
\ \right\}.
\]
\end{definition}

\begin{definition}\label{def-ue}
    Let $(\beta,p)\in \R\times\N$ and $f:\mathbb{R}^d \rightarrow \mathbb{R}$ be real-analytic at 0, we write 
    \begin{equation*}
        M(f) \curlyeqprec (\beta,p)
    \end{equation*} 
    if for $r>0$ sufficiently small, there exist $\epsilon>0$, $C>0$ and a neighbourhood $U\subset B_{\R^d}(0,r)$ of the origin such that
    \begin{equation*}
         |J(t,f+P,\psi)| \le C (1+|t|)^{\beta} \log^p(|t|+2) \|\psi\|_{C^N(U)}
    \end{equation*}
    for all $\psi\in C_0^{\infty}(U)$ and $P\in \mathcal{H}_r(\epsilon)$, where $J$ is as in \eqref{equ-J}, $N=N(h)\in \N$ and 
    $\|\psi\|_{C^N(U)}=\sup\big\{|\partial^{\gamma}\psi(\xi)|:\xi\in U,\,\gamma\in\N^d,\|\gamma\|_{\ell^1(\N^d)}\leq N\big\}.$
\end{definition}

The following theorem is a consequence of \cite[Theorem 2.1]{K84} and \cite[Theorem 0.6]{V76}.
\begin{theorem}\label{thm-twodim}
   Let $d=2$, $S$ be as in \eqref{equ-SSS} and $h_S$ be as in \eqref{equ-height},  then there exist coordinate systems that are adapted to $S$.
    Moreover,  $
    M(S)\curlyeqprec (\tau_S,\rho_S)
    $ and $
    \tau_S=-h_S^{-1}.
    $
\end{theorem}

In the sequel, for $\xi\in\R^d$, we write $M(h,\xi) \curlyeqprec (\beta,p)$ if $M(\boldsymbol{\tau}_\xi h)\curlyeqprec (\beta,p)$. Also, we write $M(h_2)\curlyeqprec M(h_1)+(\beta_2, p_2)$, if
    $
    M(h_1)\curlyeqprec (\beta_1,p_1) $ implies that $  M(h_2)\curlyeqprec (\beta_1+\beta_2,p_1+p_2).
    $
    Moreover, if $M(h_2)\curlyeqprec M(h_1)+(0,0)$, then we write $M(h_2)\curlyeqprec M(h_1)$. And if $M(h)\curlyeqprec (\beta_2,p_2)$, then we write $M(h)+(\beta_1,p_1)\curlyeqprec(\beta_1+\beta_2,p_1+p_2)$.

For a given weight 
\begin{equation}\label{equ-weight}
    \alpha=(\alpha_1,\cdots,\alpha_d),\quad\mbox{with}\quad 0<\alpha_d\le \cdots\le \alpha_1 < 1,
\end{equation}
 the one-parameter dilation is defined as
\begin{equation*}
 r^\alpha\xi:=(r^{\alpha_1}\xi_1,\cdots,r^{\alpha_d}\xi_d),\quad \forall\, r>0,\ \, \xi\in\R^d.
\end{equation*}
   \begin{definition}
     A polynomial $f$ is called $\alpha$-homogeneous of degree $\varrho\,(\geq 0)$, if 
       $$
       f(r^\alpha \xi) = r^\varrho f(\xi), \quad \forall\, r>0,\ \xi\in\R^d.
       $$
\end{definition}

\begin{definition}\label{def-space}
    Let $\mathcal{E}_{\alpha,d}$ be the set of  $\alpha$-homogeneous polynomials of degree $1$, $\mathcal{L}_{\alpha,d}$ be the linear space (over $\R$) of $\alpha$-homogeneous polynomials of degree less than $1$, and $H_{\alpha,d}$ be the set of functions real-analytic at 0 with the associated Taylor's series having the form $\sum_{\gamma\cdot\alpha>1} a_\gamma \xi^\gamma$, i.e. each monomial is $\alpha$-homogeneous of degree greater than 1. 
\end{definition}

The following lemma is from \cite[Section 2]{BCH23}, which essentially dates back to \cite{K83}.

\begin{lemma}\label{lem-nocritical+priciple+quad}
    Let $f:\R^d \rightarrow \R$ be real analytic at $0$.
    \begin{itemize}
        \item [(a)] If $\nabla f(0)\neq 0$, then for any $n\in\N$,
        $M(f)\curlyeqprec (-n,0).$ 
        
        \item [(b)] If $f\in \mathcal{E}_{\alpha,d}$ and $P\in H_{\alpha,d}$, then
        $M(f+P) \curlyeqprec M(f).$
        
        \item [(c)] If $
        g(\xi,z) = f(\xi) + \sum_{j=1}^m c_j z_j^2$
        with all $c_j\ne 0$, then
        $  M(g) \curlyeqprec M(f)+ \left(-\frac{m}{2},0\right) .$
    \end{itemize}
\end{lemma}
\begin{prop}\label{prop-Q_11 3+4}
   Let $d=5$ and $\mathbf{P}_m$ be as in \eqref{equ-PP}, it holds that
    \[
    (i) \ \  M(\mathbf{P}_3) \curlyeqprec(-2,1),\quad (ii) \ \ M(\mathbf{P}_4) \curlyeqprec \left(
    -\tfrac{11}{6},0\right).
    \]
\end{prop}
\begin{remark}
   The index is sharp and is matched with the Newton polyhedra, see \cite{BCH23} for the case $\mathbf{P}_3$ and \hyperref[Appendix]{Appendix} for the case $\mathbf{P}_4$. This proposition is the key in the proof of Theorem \ref{thm-main thm}. We postpone the proof of Proposition \ref{prop-Q_11 3+4} to Section \ref{sec-uniform esti}.
\end{remark}

\section{Proof of Theorem \ref{thm-main thm}}\label{sec-main thm}

The reader is recommended to have \cite{BCH23} at hand, since the following proof is similar to that of \cite[Theorem 1.1]{BCH23}, and we shall use the results in that paper without repeating all of the proofs here.

The strategy is as follows. For fixed $v_0$, we first study the long-time asymptotic behaviour for \eqref{equ-fundamental I} with $v=v_0$. Then we prove the same decay estimate holds uniformly under (analytic) perturbation $P$, that is, when $\phi(v_0,\cdot)$ is replaced by $\phi(v_0,\cdot)+ P$ in the integrand of \eqref{equ-fundamental I}. In our context, note that
\begin{equation}\label{equ-linper}
    \left|\phi(v,\xi)-\phi(v_0,\xi)\right|=  |(v-v_0)\cdot \xi| \le \pi^d \,|v-v_0|, \quad\xi\in\mathbb{T}^d.
\end{equation}
Therefore, the same estimate holds uniformly for $I(v,t)$ as long as $v$ belongs to some small neighborhood of $v_0$. Then it suffices to apply a finite covering since \eqref{equ-wave equ} has finite speed of propagation.

Now we begin the proof, a direct computation gives that $|\nabla \omega|<1$ on $\mathbb{T}^d\backslash\{0\}$, hence the critical points of $\phi(v,\cdot)$ only appear when $|v|<1$. Thanks to the results \cite[Proposition 2.1, Proposition 2.2, Proposition 3.10]{S98}, there exists $\mathbf{c}=\mathbf{c}(d)\in (0,1)$ such that $|G(tv,t)|\lesssim (1+|t|)^{-\frac{d}{2}}$ provided $t\in\R$ and $|v|>\mathbf{c}$.

Thus we restrict the attention to {\textbf{small}} $v$. Since $\omega$ is periodic, there exists $\eta\in C_0^\infty((-2\pi,2\pi)^d)$, $\eta(0)=1$ such that the integral in \eqref{equ-fundamental I} can be rewritten as 
\begin{align*}
    I(v,t)
    =\frac{1}{(2\pi)^d}\int_{\R^d}e^{it\phi(v,\xi)} \eta(\xi)\,\omega(\xi)^{-1} \,d\xi,\quad vt\in\Z^d,
\end{align*}
cf. e.g. \cite[Section 3]{BCH23}. 
Choosing $\chi\in C_0^\infty(\R^d)$ with support near the origin gives
\[
 (2\pi)^dI(v,t) = \int_{\R^d}e^{it\phi(v,\xi)} \eta(\xi)\omega(\xi)^{-1} \chi(\xi)\,d\xi+ \int_{\R^d}e^{it\phi(v,\xi)} \eta(\xi)\omega(\xi)^{-1} (1-\chi(\xi))\,d\xi=:I_1+I_2.
\]

By \cite[Proposition 2.3]{S98}, we know $I_1= \mathcal{O}(|t|^{-d+1})$ as $t\rightarrow\infty$. As for $I_2$, its asymptotic is determined by the critical points of the phase $\phi(v,\cdot)$. 

Let $0\le k\le d$ be an integer. Note that Hess$_\xi \phi(v,\xi)=-$Hess$\,\omega(\xi)$, we set
\[
\Sigma_k \equiv \Sigma_k(d) := \left\{\xi\in\mathbb{T}^d\backslash\{0\}: \
\begin{aligned}
    &\exists\, v\in B_{\R^d}(0,\mathbf{c})\ \mbox{such that}\ \nabla_\xi \, \phi(v,\xi)=0\\ 
    &\qquad\  \mbox{and} \ \ \mbox{corank\,Hess}_\xi \, \phi(v,\xi)=k
\end{aligned}
\ \right\}.
\]
Moreover, let $\Omega_k :=  \nabla\omega(\Sigma_k)$. By \cite[Lemma 3.1, Corollary 3.2]{BCH23}, we have (only the first quadrant $[0,\pi]^5$ is considered by symmetry):
\begin{lemma}\label{lem-dege}
    Let $d=5$, then
    \begin{itemize}
    \item[(a)] $\Sigma_k$ consists of $\xi$ with exactly ($k+1$) components equal to $\frac{\pi}{2}$ for $k=2,3,4$.  

    \item[(b)] $\Sigma_5=\emptyset$.

    \item[(c)] there exists $\mathbf{c}_*\in(\mathbf{c},1)$ such that $\cup_{k=1}^4 \Omega_k\subset B_{\R^5}(0,\mathbf{c}_*)$.

    \item[(d)] $ (\nabla \omega)^{-1}(\Omega_{4})=\Sigma_{4}$ and  $\Omega_i\cap\Omega_j=\emptyset,\,\forall\, i,j\ge 2,\ i\neq j.$ 

    \end{itemize}
\end{lemma}

Due to the compactness and Definition \ref{def-ue}, in order to obtain the uniform estimate for $I_2$, it suffices to establish local bounds in $B_{\R^5}(0,\mathbf{c}_*) \times \mathcal{U}$ and then use partition of unity, where $\mathcal{U}$ is the support of $\eta\,{\omega}^{-1}\left(1-\chi\right)$. Indeed, we have the following lemma, whose proof relies on \eqref{equ-linper} and can be found in \cite[p. 13]{BCH23}.
    
\begin{lemma}\label{lemma-inside the cone}
    For any $q_0=(v_0,\xi_0) \in B_{\R^5}(0,\mathbf{c}_*) \times \mathcal{U}$, suppose that 
    \begin{equation}\label{equ-Mxi}
        M(\phi(v_0,\cdot),\xi_0) \curlyeqprec (\beta_{q_0},p_{q_0}).
    \end{equation}
    Then 
    $
        |I_2| \lesssim (1+|t|)^{\beta}\log^{p}(2+|t|)
    $ for some $(\beta,p)$.
\end{lemma}

In fact, we only need to handle finite pairs $(\beta_{q_0},p_{q_0})$ by a partition of unity, and $(\beta,p)$ is their maximum in lexicographic order. The ``worst" index appears exactly when $q_0\in \Sigma_4\times \Omega_4$.
More precisely, to establish \eqref{equ-Mxi}, it suffices to prove:
\begin{theorem}\label{thm-4cases}
    Let $d=5$ and $q_0=(v_0,\xi_0) \in B_{\R^5}(0,\mathbf{c}_*) \times \mathcal{U}$, then
    \begin{equation*}
        M\big(\phi(v_0,\cdot),\, \xi_0\big) \curlyeqprec (\beta,p),\ \ \mbox{with}\ \   (\beta,p)= 
        \begin{cases}
            (-11/6,0), & \text{if $~q_0 \in \Omega_4\times \Sigma_4$}; \\
           (-2,1), & \text{if $~q_0 \in \Omega_3\times \Sigma_3$}; \\
            (-13/6,0), & \text{if $~q_0 \in (\Omega_2\backslash\Omega_1)\times\Sigma_2$}; \\
            (-2,0), & \text{otherwise}.
        \end{cases}
    \end{equation*}
\end{theorem}

Once proving Theorem \ref{thm-4cases}, we finish the proof of Theorem \ref{thm-main thm}. In summary, we have used Lemmas \ref{lem-nocritical+priciple+quad}, \ref{lem-dege}, \ref{lemma-inside the cone} and Theorem \ref{thm-4cases} (and hence Proposition \ref{prop-Q_11 3+4}).

\begin{proof}[Proof of Theorem \ref{thm-4cases}]
We consider each case separately and use Taylor's formula for $\phi(v_0,\xi)$ at $\xi=\xi_0$.

\noindent
\textbf{Case 1: $\boldsymbol{q_0 \in \Omega_4\times \Sigma_4}$.}\label{ssec-sigma4}
By \cite[Lemma 3.6]{BCH23}, there exists invertible linear transformation $\Phi$ which preserves the origin such that
\begin{align*}
    \phi(v,\Phi(y)+\xi_0)=c+v\cdot\xi_0+\Phi(y)\cdot(v-v_0)+\left(y_5^2+\left(\sum_{j=1}^4 y_j\right)^3-\sum_{j=1}^4 y_j^3\right)+R_1(y)
\end{align*}    
holds for $y$ near $0$, where $R_1\in H_{\boldsymbol{w_5},5}$ (recall Definition \ref{def-space}) with $\boldsymbol{w_5}=(\frac{1}{3},\cdots,\frac{1}{3},\frac{1}{2})$. Therefore, taking $v=v_0$ gives
\begin{align*}
   M\left(\phi(v_0,\cdot),\xi_0\right)\curlyeqprec M\left(\mathbf{P}_{4}\right)\curlyeqprec\left(-\frac{11}{6},0\right),
\end{align*}
where we used Lemma \ref{lem-nocritical+priciple+quad} (b)-(c) and Proposition \ref{prop-Q_11 3+4} $(ii)$.

\noindent
\textbf{Case 2: $\boldsymbol{q_0 \in \Omega_3\times \Sigma_3}$.}\label{ssec-sigma3}
By symmetries and Lemma \ref{lem-dege} (a), we can assume that $\xi_0=(\frac{\pi}{2},\frac{\pi}{2},\frac{\pi}{2},\frac{\pi}{2},\xi_*)$ with $\xi_*\ne \frac{\pi}{2}$. We use a change of variables
\[
\begin{cases}
      \xi_j = z_j+\frac{\pi}{2},\quad j=1,2,3,5;\\
         \xi_4 = z_4-z_1-z_2-z_3-z_5\sin \xi_* +\xi_*.
\end{cases}
\]
Then a direct computation yields
\begin{equation*}
    \begin{aligned}
        \phi(v_0,\xi)
        &= c+\frac{\sqrt{2}}{2\omega(\xi_0)^3} z_4^2-\frac{\cos \xi_*}{\omega(\xi_0)} z_5^2 - \frac{\sqrt{2}}{\omega(\xi_0)}(z_1+z_2)(z_1+z_3)(z_2+z_3)+ R_2(z), 
    \end{aligned}
\end{equation*}
where $R_2\in H_{\boldsymbol{w_4},5}$ with $\boldsymbol{w_4}=(\frac{1}{3},\frac{1}{3},\frac{1}{3},\frac{1}{2},\frac{1}{2})$.
Note that
    \begin{equation}\label{equ-x1x2x31}
        (z_1+z_2)(z_2+z_3)(z_1+z_3)=\frac{1}{3}\left((z_1+z_2+z_3)^3 - z_1^3+z_2^3+z_3^3\,\right).
    \end{equation}
         By Lemma \ref{lem-nocritical+priciple+quad} (b)-(c) and Proposition \ref{prop-Q_11 3+4} $(i)$, we have
\[
M(\phi(v_0,\cdot),\xi_0)\curlyeqprec M(\mathbf{P}_3) \curlyeqprec (-2,1).
\]

\noindent
\textbf{Case 3: $\boldsymbol{q_0\in(\Omega_2\backslash\Omega_1)\times \Sigma_2}$.}\label{ssec-sigma2}
In this case, a direct computation shows that the zero-eigenvectors of $\mbox{Hess}_\xi \phi(q_0)$ are $\boldsymbol{\gamma_1}=(1,-1,0,0,0)^T$ and $\boldsymbol{\gamma_2}=(1,1,-2,0,0)^T.$
Therefore, we let the matrix $\mathbb{A}=(\boldsymbol{\gamma_1},\boldsymbol{\gamma_2},\boldsymbol{e_{3}},\boldsymbol{e_{4}},\boldsymbol{e_{5}})$. By a change of variables $\xi=\mathbb{A} z+ \xi_0$ and then use a rotation in $\{z_3,z_4,z_5\}$, we get
\begin{equation*}
    \phi(v_0, \mathbb{A} z+ \xi_0)=c +c_1 (z_2^3 - z_1^2 z_2 ) + c_2 z_3^2+c_3 z_4^2+c_4 z_5^2+ R_3(z),
\end{equation*}
where $c_1 c_2 c_3 c_4\ne 0$ and $R_3 \in H_{\boldsymbol{w_3},5}$ with $\boldsymbol{w_3}=\left(\frac{1}{3},\frac{1}{3},\frac{1}{2},\frac{1}{2},\frac{1}{2}\right)$. Since $z_2^3-z_1^2 z_2$ has $D_4^-$ type singularity, Lemma \ref{lem-nocritical+priciple+quad} and Theorem \ref{thm-twodim} give that
\begin{align*}
    M\left( \phi(v_0,\cdot), \xi_0 \right)\curlyeqprec M\left( \mathbf{P}_2  \right)\curlyeqprec \left(-\frac{13}{6},0\right).
\end{align*}

\noindent
\textbf{Case 4: otherwise.}\label{ssec-sigma other}
Note that the rank of $\phi$ at $q_0$ is at least 4, the splitting lemma (cf. \cite{M21}) and Lemma \ref{lem-nocritical+priciple+quad} (c) imply a rough upper bound $\mathcal{O}(|t|^{-2})$, which meets our needs already.
\end{proof}

\section{Proof of Proposition \ref{prop-Q_11 3+4}}\label{sec-uniform esti}

\subsection{Preparation}
We first give a little notation to state the algorithm (Theorem \ref{thm-algo}) in Karpushkin \cite{K83} and some simplification (Lemma \ref{lem-S}), making it convenient to verify the conditions in his result. 

Let $d\ge 2$, weight $\alpha$ be as in \eqref{equ-weight} and $ h\in \mathcal{E}_{\alpha,d}$. We set
\begin{gather*}
    \hh := h|_{\Sl},\quad\mbox{where}\quad\Sl = \left\{x\in\R^d:|x|_{\alpha,l}:=\left(x_1^{\frac{l}{\alpha_1}}+\cdots+x_d^{\frac{l}{\alpha_d}}\right)^{\frac{1}{l}}=1\right\},\\
    \mbox{while} \ \ 0<l\in\mathbb{Q} \ \ \mbox{such that all}  \ l/\alpha_j \  \mbox{are even numbers}.
\end{gather*}
Notice that $h(0)=\nabla h(0)=0$. Let
$$\mathcal{Z}_h = \left\{\sss\in\Sl: \hh(\sss)=\mathrm{d}\hh(\sss)=0\right\},$$
where $\mathrm{d}\hh(\sss)$ is the differential of $\hh$ at $\sss$. Moreover,
 let $\mathcal{I}_{\nabla h}$ be the Jacobi ideal of $h$ (cf. e.g. \cite[p. 51]{M21}). We have the following definition, which is from \cite{K83}.
\begin{definition}\label{def-versal}
A subspace $\mathcal{B}\subset \CC$ is said to be lower $(h,\alpha)$-versal,  if 
\[
\left(\mathcal{I}_{\nabla h} \cap \CC\right) \oplus \mathcal{B} = \CC.
\]
\end{definition}

Recall that $\mathcal{T}(h)$ is defined in \eqref{equ-series} and $\boldsymbol{\tau}_\xi$ $(\xi\in\R^d)$ is the translation. Using \cite[Proposition 4 on p. 1182, Lemma 21 on p. 1184 ]{K83}, we have
\begin{lemma}\label{lem-qua}
    Let $h\in \mathcal{E}_{\alpha,d}$ and the first order partial derivatives of $h$ be linearly independent. Then for any lower $(h,\alpha)$-versal subspace $\mathcal{B}\ne 0$, $g \in \mathcal{B}\backslash\{0\}$ and any critical point $\bb$ of $h+g$, there exists monomial $\iota\in \CC\backslash\{0\}$ such that $\mathcal{T}(\iota)\in \mathcal{T}(\boldsymbol{\tau}_\mathfrak{b}(h+g))$.
\end{lemma}

Note that Definition \ref{def-ue} carries over to real analytic manifolds. In the sequel, we  write $\|\alpha\|_1 :=\|\alpha\|_{\ell^1(\N^d)} (= \sum_{j=1}^d \alpha_j)$.
\begin{theorem}[Theorem 1 in \cite{K83}]\label{thm-algo}
    Let $h\in \mathcal{E}_{\alpha,d}$ and the following two conditions hold.
    \begin{itemize}
        \item[(a)] There exists a lower $(h,\alpha)$-versal subspace $\mathcal{B}\ne 0$, such that for any $g\in \mathcal{B}\backslash\{0\}$ and any critical point $\mathfrak{b}$ of $h+g$, it holds that
        \[
        M(h+g,\mathfrak{b}) \curlyeqprec (\beta_1,p_1).
        \]
        \item[(b)] $\mathcal{Z}_h\ne \emptyset$, and 
        \[
        M(\hh, \sss)\curlyeqprec (\beta_2,p_2),\quad \forall \, \sss\in\mathcal{Z}_{h}.
        \] 
    \end{itemize}
    Then  $M(h)\curlyeqprec (\beta,p)$, where
    \[
   (\beta,p)=\left\{
    \begin{aligned}
        &\mathrm{max}\{(\beta_1,p_1),(\beta_2,p_2),(-\|\alpha\|_1,0)\}, &\mbox{if}\ \ \|\alpha\|_1 + \beta_2 \ne 0.\\
        &\mathrm{max}\{(\beta_1,p_1),(\beta_2,p_2+1)\}, &\mbox{if}\ \ \|\alpha\|_1 + \beta_2 = 0.
    \end{aligned}\right.
    \]
    If $\mathrm{(a)}$ holds and $\mathcal{Z}_h= \emptyset$, then
    \[
    M(h) \curlyeqprec \mathrm{max}\{(\beta_1,p_1),(-\|\alpha\|_1,0)\}.
    \]
    Here the maximum is taken in the lexicographic order.
\end{theorem}

By Lemma \ref{lem-qua}, we know condition (a) can be simplified if the first order partial derivatives of $h$ are linearly independent, see Section \ref{ssec-prop(i)}.

To simplify condition (b), we define the projection
\begin{align*}
    \pi_{\alpha,l}:\quad \R^d\backslash\{0\}\ &\longrightarrow \ \Sl \\
    \xi \ &\longmapsto \ \mathbf{r}^{\alpha}\xi, \ \quad \mbox{with}\quad \mathbf{r}= \frac{1}{|\xi|_{\alpha,l}},
\end{align*}
as well as the set 
\[
\mathcal{C}_h := \{\xi\in\R^d\backslash\{0\}: \nabla h(\xi)=0\}.
\]
Then we have the following relations.

\begin{lemma}\label{lem-S}
  If $h\in\mathcal{E}_{\alpha,d}$ and $\mathcal{C}_h\ne \emptyset$, we have
    \begin{equation}\label{equ-Proj}
         \pi_{\alpha,l}(\mathcal{C}_h)= \mathcal{Z}_h.
    \end{equation}
    Moreover, if all $\alpha_j$ are equal, then
\begin{equation}\label{equ-RANK}
          \mathrm{rank \, Hess \,} \hh (\pi_{\alpha,l}(\xi_0)) = \mathrm{rank\, Hess \,} h(\xi_0), \quad\forall\, \xi_0\in \mathcal{C}_h.
    \end{equation}
\end{lemma}
\begin{proof}
We first prove \eqref{equ-Proj}. By the $\alpha$-homogeneity of $h$, we have
\begin{align}\label{equ-partial}
    \partial_j h(r^\alpha \xi) = r^{1-\alpha_j} \partial_j h(\xi)\ \ \mbox{for}\ \  j=1,\cdots,d, \quad\mbox{and} \quad  \mathfrak{E}_{\alpha}(h)(\xi) = h(\xi),
\end{align}
where $\mathfrak{E}_{\alpha}=\sum_{j=1}^{d} \alpha_j \xi_j \partial_j$ is the Euler vector field for $\alpha$. Thus, it follows easily that $\pi_{\alpha,l}(\mathcal{C}_h)\subset \mathcal{Z}_h$. 

To see the reverse, taking $\sss= (\sss_1,\cdots,\sss_d)=:(\sss',\sss_d)\in\mathcal{Z}_h$, we assume that $\sss_d> 0$, and use the chart $\xi' \longmapsto (\xi', \Pi(\xi'))$, where
\begin{gather*}
    \xi'=(\xi_1,\dots,\xi_{d-1})\in B_{\R^{d-1}}(0,1)\ \
    \mbox{and} \ \ \Pi(\xi') = \left(1-\left(\xi_1^{\frac{l}{\alpha_1}} + \cdots + \xi_{d-1}^{\frac{l}{\alpha_{d-1}}}\right)\right)^{\frac{\alpha_d}{l}}.
\end{gather*}
Then
$
\hh(\xi') = h(\xi',\Pi(\xi'))
$
and $\mathrm{d}\hh(\sss)=0$ gives
\begin{equation}\label{equ-Pi}
    \partial_j h(\sss) = \frac{\alpha_d}{\alpha_j}\, \sss_j^{\frac{l}{\alpha_j}-1}  \partial_d h(\sss) \,\Pi(\sss')^{1-\frac{l}{\alpha_d}}\  \mbox{for}\ j=1,\cdots,d-1.
\end{equation}

Combining \eqref{equ-Pi} and the second equality in \eqref{equ-partial} with $\xi=\sss$, we get (note that $\sss_d= \Pi(\sss')$) 
\[
\alpha_d\, \partial_d h(\sss)\,\Pi(\sss')^{1-\frac{l}{\alpha_d}} = h(\sss).
\]
Since $\sss\in\mathcal{Z}_h$, we have $h(\sss)=0$. Then $\partial_d h(\sss)=0$, which yields $\nabla h(\sss)=0$ by \eqref{equ-Pi} again. Therefore, $\sss\in \mathcal{C}_h$ and \eqref{equ-Proj} is proved.

Now we prove \eqref{equ-RANK}. By the proof of \eqref{equ-Proj}, we know $\mathcal{Z}_h \subset \mathcal{C}_h$. It suffices to consider $\xi_0\in \mathcal{Z}_h$ since $h\in \mathcal{E}_{\alpha,d}$. Taking $\sss\in\mathcal{Z}_h$, we can choose $l=2\alpha_1$ and then $\mathbb{S}^{d-1}_{\alpha,l}=\mathbb{S}^{d-1}$, the standard sphere in $\R^d$. Moreover, we  assume $\sss_d>0$, then the standard chart is
$$\Pi(\xi')=(\xi',\sqrt{1-|\xi'|^2}),\quad \xi'\in B_{\R^{d-1}}(0,1).$$

Since $\nabla h(\sss)=0$, a direct computation gives
\begin{gather*}
    \mbox{Hess}\, \hh(\sss) = \mathbf{H}\, \mbox{Hess}\, h(\sss)\,\mathbf{H}^T,\quad \mbox{where}\ \ \mathbf{H} = \left(
\begin{array}{cc}
    \mathbb{I}_{d-1} & \nabla_{\xi'} \Pi(\sss')
\end{array}
\right),\\
\mbox{and} \ \  \mathbb{I}_{d-1} \ \ \mbox{is the}\ \ (d-1)\times (d-1) \ \ \mbox{identity matrix.}
\end{gather*}
Notice that  $\partial_j h(r^\alpha \sss)= r\, \partial_j h(\sss)=0$, $\forall\,r>0$. Taking derivative in $r$ gives
\[
\frac{\partial^2 h(\sss)}{\partial \xi_{j} \partial \xi_{d}} = \partial_1 
 \Pi(\sss')\,\frac{\partial^2 h(\sss)}{\partial \xi_{j} \partial \xi_{1}}+\cdots+ \partial_{d-1}\Pi(\sss')\, \frac{\partial^2 h(\sss)}{ \partial \xi_{j} \partial\xi_{d-1} },\quad j=1,\cdots,d.
\]
Therefore, we simplify Hess$\, h(\sss)$ and get that
\[
\mbox{Hess}\, \hh(\sss) \, = \, \mathbf{H}\mathbf{H}^T 
\left(
\begin{array}{ccc}
    \frac{\partial^2 h(\sss)}{\partial \xi_{1}^2} & \cdots& \frac{\partial^2 h(\sss)}{\partial \xi_{1} \partial \xi_{d-1}} \\
    \vdots & \ddots & \vdots\\
    \frac{\partial^2 h(\sss)}{\partial \xi_{d-1} \partial \xi_{1}} & \cdots & \frac{\partial^2 h(\sss)}{\partial \xi_{d-1}^2 }
\end{array}
\right)
\mathbf{H}\mathbf{H}^T.
\]
Since $\mathbf{H}\mathbf{H}^T =  \mathbb{I}_{d-1}+ \nabla_{\xi'} \Pi(\sss')(\nabla_{\xi'} \Pi(\sss'))^T$ is non-singular, we finish the proof of \eqref{equ-RANK}.
\end{proof}

With all the tools in hands, we are in a position to prove Proposition \ref{prop-Q_11 3+4}.

\subsection{Proof of Proposition \ref{prop-Q_11 3+4} $(i)$}\label{ssec-prop(i)}
    Take $\gamma_1=(\frac{1}{3},\frac{1}{3},\frac{1}{3})$.
    By \eqref{equ-x1x2x31}, a change of coordinates and Lemma \ref{lem-nocritical+priciple+quad} (c), it suffices to prove
    \[
    M(h)\curlyeqprec (-1,1),\quad\mbox{with}\ \ h(z)=z_1z_2z_3.
    \]
    Since $ (z_2z_3,z_1z_3,z_1z_2)$ are linearly independent, Lemma \ref{lem-qua} gives that for any lower $(h,\gamma_1)$-versal subspace $\mathcal{B}\ne 0$ and $g\in \mathcal{B}\backslash\{0\}$, the rank of $h+g$ at any critical point is nonzero. If the rank of $h+g$ at critical point $\bb$ is at least 2,  the splitting lemma and Lemma \ref{lem-nocritical+priciple+quad} (c) yield
    \[
    M(h+g,\bb)\curlyeqprec (-1,0).
    \]
    If the rank of $h+g$ at critical point $\bb$ is 1, it holds that (since $\mathcal{B}\subset \mathcal{L}_{\gamma_1,3}$)
    \begin{equation}\label{equ-unfold}
        \boldsymbol{\tau}_{\bb}(h+g)(z)=c_0+c (z_1+ a_2 z_2 + a_3 z_3)^2 + h,\quad\mbox{for some}\ c\ne 0, \ c_0,a_2,a_3\in\R.
    \end{equation}
    A change of variables $y_1=z_1+ a_2 z_2 + a_3 z_3 + \frac{z_2 z_3}{2 c}$, $y_2=z_2,y_3=z_3$ gives
    \[
      \boldsymbol{\tau}_{\bb}(h+g)(z)= c_0+c\, y_1^2 - y_2 y_3 (a_2 y_2 + a_3 y_3)-\frac{1}{4c}\, y_2^2 y_3^2.
    \]
    By Lemma \ref{lem-nocritical+priciple+quad} (c) again, it suffices to prove that 
       \[
       M\left( y_2 y_3 (a_2 y_2 + a_3 y_3)+\frac{1}{4c}\, y_2^2 y_3^2\right) \curlyeqprec \left(-\frac{1}{2},1\right).
       \]
       We consider two cases $a_2= a_3 =0$ and $a_2^2+ a_3^2>0$. It suffices to prove
       \[
       M(y_1^2 y_2)\curlyeqprec \left(-\frac{1}{2},0\right) \ \  \mbox{and}\   \ M(y_1^2 y_2 ^2) \curlyeqprec \left(-\frac{1}{2},1\right),\  \ \mbox{respectively}.
       \]
       These two estimates can be derived by, for instance, a combination of Theorem \ref{thm-twodim} and Proposition \ref{prop-adapt} (alternatively, one can use Theorem \ref{thm-algo} again).
       In conclusion, we verify condition (a) in Theorem \ref{thm-algo} with $(\beta_1,p_1) = (-1,1)$.

       Next, by Lemma \ref{lem-S} and the formula of Hess\,$h(z)$, we know
       \[
       \mathcal{Z}_h = \{\pm \boldsymbol{e_j}: \, j=1,2,3\}\subset \mathbb{S}^2_{\gamma_1,\frac{2}{3}}\, (= \mathbb{S}^2).
       \]
      By Lemma \ref{lem-S} again, we know each critical point of $\hh$ is non-degenerate. Therefore, $(\beta_2,p_2)=(-1,0)$.

       As a result, Theorem \ref{thm-algo} gives the desired estimate $M(h)\curlyeqprec (-1,1)$.

\subsection{Proof of Proposition \ref{prop-Q_11 3+4} $(ii)$}
Now we deal with $\mathbf{P}_4$, which has $O$ type singularity, see e.g. \cite[p. 252]{AGV12}.
To begin with, we give two elementary lemmas.
\begin{lemma}\label{lem-binary4}
    Suppose that
    $$
    f(x_1,x_2) = \sum_{k=0}^4  a_k \, x_1^{4-k} x_2^{k},\quad \mbox{with}\ \ (a_0,\cdots,a_4)\in\R^5\backslash\{0\}.
    $$
   Then $M(f) \curlyeqprec (-\frac{1}{4},0)$ if and only if $f(x_1,1)$ or $f(1,x_2)$ has a real root with multiplicity four. Otherwise, we have $M(f) \curlyeqprec (-\frac{1}{3},0)$.
\end{lemma}

\begin{proof}
    Assume that $a_0\ne 0$. We discuss the multiplicity for the roots of $f(x_1,1)$ in $\C$ by a similar argument in \cite[p. 85]{M21}. Under proper linear transforms $f$ can be reduced to one of the following forms:
    \begin{align}\label{equ-forms}
        \begin{split}
             x_1^4,&\ x_1^3 x_2, \ x_1^2 x_2^2,\ x_1^2(x_1^2 + x_2^2),\ x_1^2 x_2 (x_1+x_2),\\
              (x_1^2 + x_2^2)(x_1^2 + a_1  x_1&  x_2 + a_2 x_2^2), \ a_1^2 + a_2^2 \ne 0\,;\  x_1x_2(x_1^2+ a_1 x_1x_2+ a_2 x_2^2), \ a_2\ne 0.
        \end{split}
    \end{align}
    
By Proposition \ref{prop-adapt}, we check that in each case (except $x_1^4$), the coordinate system $\{x_1,x_2\}$ is adapted. Consequently, Van der Corput lemma and Theorem \ref{thm-twodim} yield
    \begin{align*}
        & M(x_1^4)\curlyeqprec\left(-\frac{1}{4},0\right),\ \ \, \qquad\ \; \quad M(x_1^3x_2) \curlyeqprec \left(-\frac{1}{3},0\right), \\
        & M(x_1^2(x_1^2+ x_2^2)) \curlyeqprec \left(-\frac{1}{2},1\right),\,\ M(x_1^2x_2(x_1+x_2))\curlyeqprec \left(-\frac{1}{2},1\right),     
    \end{align*}
    while for last two cases in \eqref{equ-forms}, the index is $(-\frac{1}{2},0)$.
We complete the proof.
\end{proof}

\begin{lemma}\label{lem-coincide}
    Let $(m_1,m_2,m_3)\in\R^3\backslash\{0\}$ and $(m_4,m_5,m_6)\in \R^3\backslash\{0\}$,  
    \begin{equation*}
        f(r)=m_1 r^2+m_2 r+m_3,\ \  g(r)=m_4 r^2+m_5 r +m_6.
    \end{equation*}
    If there exists a constant $c\in\R\backslash\{0\}$ such that $f^2 + c\, g^2 \not\equiv 0$ and has a real root with multiplicity four, then there exists $c_0\in\R\backslash\{0\}$ such that $f=c_0\, g$.
\end{lemma}
\begin{proof}
Without loss of generality, we assume
    \[
    f(r)^2 + c g(r)^2 \equiv s( r - r_0)^4\quad \mbox{for some}\ s\in \R\backslash\{0\}.
    \]
    
    If $c>0$, then $s>0$ and $r_0$ is the root of $f$ and $g$, the conclusion is obvious.

     If $c<0$, since
     \[
     f(r)^2 + c g(r)^2  = \left(f(r)+ \sqrt{-c}\, g(r)\right)\left(f(r)-\sqrt{-c}\, g(r)\right),
     \]
     we know both $f + \sqrt{-c}g$ and $f - \sqrt{-c}g$ have root $r_0$ with multiplicity 2. Then the proof is finished.
\end{proof}

Let 
$
    \PP(z):=(\sum_{j=1}^4 z_j)^3-\sum_{j=1}^4 z_j^3
$ for $z\in\R^4$.
  It suffices to show $M(\PP)\curlyeqprec (-\frac{4}{3},0)$ by Lemma \ref{lem-nocritical+priciple+quad} (c). We apply Theorem \ref{thm-algo}.
Since
\begin{equation*}
    \sum_{j=1}^4 a_j \partial_j \PP(z) \equiv 0 \quad\Longrightarrow\quad \left(\sum_{j=1}^4 a_j\right)\left(\sum_{j=1}^4 z_j\right)^2 \equiv \sum_{j=1}^4 a_j z_j^2,
\end{equation*}
the coefficient of $z_1 z_2$ must be zero. So $\sum_{j=1}^4 a_j =0$, which yields $a_j=0$ for all $j$ as well. Therefore, the partial derivatives of first order of $\PP$ are linearly independent.

Moreover, $\mathcal{Z}_{\PP}=\emptyset$. Indeed, the stationary equation $\nabla \PP (z)=0$ gives 
\begin{equation*}
\left(\sum_{j=1}^4 z_j\right)^2=z_1^2=z_2^2=z_3^2=z_4^2,
\end{equation*}
which has only zero solution. Then we use Lemma \ref{lem-S} to get the conclusion.

We are left with {\textbf{condition $\mathrm{(a)}$ of Theorem \ref{thm-algo}}.

Let weight $\gamma_2=(\frac{1}{3},\frac{1}{3},\frac{1}{3},\frac{1}{3})$, $\|\gamma_2\|_1=\frac{4}{3}$.  By Lemma \ref{lem-qua}, for any lower $(\PP,\gamma_2)$-versal subspace $\mathcal{B}$ and $g\in  \mathcal{B}\backslash\{0\}$, the rank of $\PP+g$ at its critical point is at least 1. Since $\frac{3}{2}>\frac{4}{3}$, it suffices to consider the cases that the rank equals to 1 or 2, by Lemma \ref{lem-nocritical+priciple+quad} (c) and the splitting lemma.

\textbf{(1) The rank 1 case.}
Note that $g\in \mathcal{L}_{\gamma_2,4}$, by symmetry we can assume that at the critical point $\PP+g$ is the 5-parameter deformation (which is the counterpart of \eqref{equ-unfold}):
\begin{equation*}
   \mathcal{F}_1 = \mathcal{F}_1 (z,\boldsymbol{a},c,c_0) :=  c_0+c(z_1+a_2z_2+a_3z_3+a_4z_4)^2+\PP,
\end{equation*}
 where $\boldsymbol{a}=(a_2,a_3,a_4)\in\R^3$, $c_0\in\R$ and $c \neq 0$. 
 
It suffices to show $M(\mathcal{F}_1)\curlyeqprec (-\frac{4}{3},0)$. Let $\widetilde{\gamma_2}=(\frac{1}{2},\frac{1}{3},\frac{1}{3},\frac{1}{3})$. We use a change of variables $y_1=z_1+a_2z_2+a_3z_3+a_4z_4$ and $y_j=z_j$ for $j=2,3,4$. A direct calculation shows  
 \[
 \mathcal{F}_1 = c_0+ c\left(y_1+\frac{3}{2c}(y_2+y_3+y_4)(y_2+ y_3 + y_4 + 2T)\right)^2+ \mathcal{F}_3 + \mathcal{F}_4+ \mathcal{F}_5,
 \]
 where $T=-a_2y_2-a_3y_3-a_4y_4$, $\mathcal{F}_5\in H_{\widetilde{\gamma_2},4}$ and
 \begin{align}\label{equ-F3F4}
 \begin{split}
      \mathcal{F}_3 &= \mathcal{F}_3(y_2,y_3,y_4)= (y_2+y_3+y_4+T)^3-(T^3+y_2^3+y_3^3+y_4^3), \\
       \mathcal{F}_4 &=\mathcal{F}_4(y_2,y_3,y_4)= -\frac{9}{4c}(y_2+y_3+y_4)^2(y_2+ y_3 + y_4 + 2T)^2.
 \end{split}
 \end{align}

 Using a change of coordinates and  Lemma \ref{lem-nocritical+priciple+quad} (b)-(c), it suffices to prove
 \begin{equation}\label{equ-F34}
     M(\mathcal{F}_3+\mathcal{F}_4)\curlyeqprec \left(-\frac{5}{6},0\right).
 \end{equation}
 
We first focus on the case w.r.t. $\boldsymbol{a}$, in which 
 $
 \frac{\partial \mathcal{F}_3}{\partial y_j} \equiv 0
 $ 
 for some $j$. By symmetry we may assume that $j=2$. Then a direct computation yields that this happens when
 \begin{equation}\label{equ-atype}
      \boldsymbol{a} = (0,1,1) \ \ \mbox{or} \ \ (1,0,0).
 \end{equation}
 In these cases $\mathcal{F}_3=3y_3y_4 (y_3 + y_4)$, which has $D_4^-$ type singularity, and 
 \[
 \mathcal{F}_3+ \mathcal{F}_4  =-\frac{9}{4c}(y_2+y_3+y_4)^2(y_2-y_3-y_4)^2+3y_3y_4 (y_3 + y_4).
 \]
 Consider the weight $ (\frac{1}{4},\frac{1}{3},\frac{1}{3})$, by Lemma \ref{lem-nocritical+priciple+quad} (b), in this case to prove \eqref{equ-F34} it suffices to establish the following estimate.

 \begin{lemma}\label{lem-MM}
     $M\big(z_1^4 + z_2 z_3 (z_2+ z_3)\big)\curlyeqprec (-\frac{11}{12},0).$
 \end{lemma}

\begin{proof}
Using a change of variables $z_1=w_1, z_2+ z_3 =w_2, z_2-z_3=w_3$, the polynomial can be written as a unimodal germ which has $U_{12}$ type singularity  in the terminology of Arnold, see \cite[p. 273]{AGV12}. The uniform estimate associated to this phase has been studied by Karpushkin with index $(-\frac{11}{12},0)$, which is matched with the oscillation index, see \cite{K89} and the reference therein.
\end{proof}

Thus, we can assume that $\frac{\partial \mathcal{F}_3}{\partial y_j} \not\equiv 0$ for all $j$. 
\begin{lemma}\label{lem-F3}
    Let $\mathcal{F}_3$ be as in \eqref{equ-F3F4}. If $\frac{\partial \mathcal{F}_3}{\partial y_j} \not\equiv 0$ for $j=2,3,4$, then $M(\mathcal{F}_3) \curlyeqprec (-\frac{5}{6},0)$.
\end{lemma}

Once proving Lemma \ref{lem-F3}, using Lemma \ref{lem-nocritical+priciple+quad} (b) again with weight $(\frac{1}{3},\frac{1}{3},\frac{1}{3})$, we know $M(\mathcal{F}_3+ \mathcal{F}_4)\curlyeqprec M(\mathcal{F}_3)$, and we shall get that $M(\mathcal{F}_1)\curlyeqprec (-\frac{4}{3},0)$ in the rank 1 case.

\begin{proof}[Proof of Lemma \ref{lem-F3}]
    We shall apply Theorem \ref{thm-algo} on $\mathcal{F}_3$. We begin with the linear independence for the partial derivatives.
Since
\begin{equation*}
    \frac{\partial \mathcal{F}_3}{\partial y_k} = 
    3(y_2+y_3+y_4+T)^2-3a_k(y_2+y_3+y_4)(2T+y_2+y_3+y_4)-3y_k^2,\ \ k=2,3,4.
\end{equation*}
Let $\lambda_2 \frac{\partial \mathcal{F}_3}{\partial y_2}+\lambda_3\frac{\partial \mathcal{F}_3}{\partial y_3}+\lambda_4\frac{\partial \mathcal{F}_3}{\partial y_4}=0$ and display the coefficient of each monomial, we get:
\begin{equation*}
    \begin{aligned}
      &A (1-a_j)^2  = B(1-2a_j)^2 + \lambda_j,\ \ j=2,3,4, \\
       & A (1-a_j)(1-a_i) = B(1-2a_j)(1-2a_i),\ j\ne i,\ \ 2\le i,j\le 4,
    \end{aligned}
\end{equation*}
where
$
A=\sum_{j=2}^4 \lambda_j$ and  $B=\sum_{j=2}^4 \lambda_j a_j.$ A discussion on whether $A$ (or $B$) equals to 0 gives that $\lambda_j=0$ for $j=2,3,4$. 

Now we verify condition (a) of Theorem \ref{thm-algo} for $\mathcal{F}_3$. By Lemmas \ref{lem-nocritical+priciple+quad} (c), it suffices to consider the rank 1 case. Similar to the previous proof, it suffices to prove 
\[
M(\widetilde{\mathcal{F}_3})\curlyeqprec(-\frac{5}{6},0),\quad\mbox{where}\ \ \widetilde{\mathcal{F}_3} =  c_* (y_2+b_3y_3+b_4y_4)^2+\mathcal{F}_3,\ c_*\ne 0.
\]
  By a change of variables  $\sigma_2=y_2+b_3y_3+b_4y_4,\,\sigma_3=y_3,\,\sigma_4=y_4$, we have
\begin{equation*}
    \widetilde{\mathcal{F}_3} = c_* \sigma_2^2 + 3\sigma_2 Q_2 +Q_3 +\mathcal{F}_6 = c_*\left(\sigma_2+\frac{3}{2c_*}Q_2\right)^2+Q_3-\frac{9}{4c_*}Q_2^2+\mathcal{F}_6,
\end{equation*}
where $\mathcal{F}_6 \in H_{\gamma_3,3}$ with  $\gamma_3 = (\frac{1}{2},\frac{1}{4},\frac{1}{4})$, and 
\begin{equation*}
    \begin{aligned}
        Q_2
         &=Q_2(\sigma_3,\sigma_4)= (1-a_2)\left((1-a_3-b_3+a_2b_3)\sigma_3+(1-a_4-b_4+a_2b_4)\sigma_4\right)^2\\
        &+a_2\left((a_3-a_2b_3)\sigma_3+ (a_4-a_2b_4)\sigma_4\right)^2-(b_3\sigma_3+b_4\sigma_4)^2,\\
        Q_3
        &=Q_3(\sigma_3,\sigma_4)=\left((1-a_3-b_3+a_2b_3)\sigma_3+(1-a_4-b_4+a_2b_4)\sigma_4\right)^3\\
        &+\left((a_3-a_2b_3)\sigma_3+ (a_4-a_2b_4)\sigma_4\right)^3+(b_3\sigma_3+b_4\sigma_4)^3-\sigma_3^3-\sigma_4^3.
    \end{aligned}
\end{equation*}

Note that both $Q_2$ and $Q_3$ should to be taken into account since, as we shall see, in some cases $Q_3$ vanishes and the contribution would come from $Q_2^2$.

After a change of variables $\widetilde{\sigma}_2=\sigma_2+\frac{3}{2c_*}Q_2$ with $\sigma_3$, $\sigma_4$ fixed, and using Lemma \ref{lem-nocritical+priciple+quad} (c), it remains to show
\begin{equation}\label{equ-Q23}
M\left(Q_3-\frac{9}{4c_*} Q_2^2\right) \curlyeqprec \left(-\frac{1}{3},0\right).    
\end{equation}

Our strategy is to discuss whether $Q_3$ vanishes. If $Q_3 \equiv 0$, then we show $Q_2^2$ has desired estimate.
If $Q_3 \not\equiv 0$, by the properties of binary cubic forms (see e.g. \cite[p. 85]{M21}) and the Van der Corput lemma, we also get the desired result.

Through an elementary (but tedious) computation on the coefficients of the monomials, we see that if $Q_3\equiv 0$, it suffices to consider the following two cases (other possible cases are slightly but not essentially different):

(1) $a_2=a_4 \neq 0,\,a_3=1,\,b_3=0,\,b_4=1$, and
$
   Q_3-\frac{9}{4c_*} Q_2^2=-\frac{9}{4c_*}(a_2\,\sigma_3^2-\sigma_4^2)^2.
$

(2) $a_2-a_4 = b_3 \neq 0,\,b_4=a_3-a_2 b_3=1,$ and $b_3$ is the real root of the equation $r^2-r-1=0$. In this case,
\begin{align*}
    Q_3-\frac{9}{4c_*} Q_2^2 
    & = -\frac{9}{4c_*}\left( -a_2b_3 \sigma_3^2+
        2(1-a_2)\sigma_3\sigma_4+
        b_3\sigma_4^2\right)^2.
\end{align*}

 In both cases, $Q_3(1,\sigma_4)- \frac{9}{4c}Q_2^2(1,\sigma_4)$ cannot process a root with multiplicity 4. Then we get \eqref{equ-Q23} by Lemma \ref{lem-binary4}.  
 In conclusion, we verify condition (a) of Theorem \ref{thm-algo} with $(\beta_1,p_1)=(-\frac{5}{6},0)$.

 Now we turn to condition (b) of Theorem \ref{thm-algo} for $\mathcal{F}_3$.

 By Lemma \ref{lem-S}, it suffices to show that for any possible critical point $x_0\,(\ne 0)$ of $\mathcal{F}_3$, we have rankHess$\mathcal{F}_3 (x_0)=2$. 
 We argue by contradiction that  $x_0=(w_1,w_2,w_3)$ is a critical point such that rankHess$F_3 (x_0)\le 1$, then the three principal minors in Hess$\mathcal{F}_3(x_0)$ are singular. Combining this with the fact $\nabla \mathcal{F}_3(x_0)=0$, by a direct computation we get the following six equations in $(a_2,a_3,a_4,w_1,w_2,w_3)$:
\begin{gather*}
          (1-a_2)S^2= w_1^2  -a_2 T^2,\quad\ (1-a_3)S^2 =w_2^2 -a_3 T^2,\quad\ (1-a_4)S^2=w_3^2 -a_4 T^2,\\
        \big( (1-a_2)^2 S -a_2^2 T-w_1\big)\big((1-a_3)^2 S -a_3^2 T-w_2\big)=\big((1-a_2)(1-a_3)S-a_2a_3T\big)^2,\\
        \big( (1-a_2)^2 S -a_2^2 T-w_1\big)\big((1-a_4)^2 S -a_4^2 T-w_3\big)=\big((1-a_2)(1-a_4)S-a_2 a_4 T\big)^2,\\
        \big( (1-a_3)^2 S -a_3^2 T-w_2\big)\big((1-a_4)^2 S -a_4^2 T-w_3\big)=\big((1-a_3)(1-a_4)S-a_3 a_4 T\big)^2,
\end{gather*}
where $S:= T+ w_1 + w_2 + w_3$.
Since $x_0\ne 0$, we claim that under the assumption of Lemma \ref{lem-F3}, this system has no solution. In fact, we solve $(a_2,a_3,a_4)$ from the first three equations and put them into the left three equations. Then we get a system of $(T,w_1,w_2,w_3)$,
\begin{gather*}
    S^3 = T^3 + w_1^3 + w_2^3 + w_3^3,\\
    S(w_2(T^2-w_1^2)^2+ w_1(T^2-w_2^2)^2)+ (w_1^2-w_2^2)^2 ST = w_1 w_2(T^2-S^2)^2\\
    + T((w_1^2-S^2)^2 w_2 + (w_2^2-S^2)^2 w_1),\\
    S(w_3(T^2-w_1^2)^2+ w_1(T^2-w_3^2)^2)+ (w_1^2-w_3^2)^2 ST = w_1 w_3(T^2-S^2)^2\\
    + T((w_1^2-S^2)^2 w_3 + (w_3^2-S^2)^2 w_1),\\
    S(w_3(T^2-w_2^2)^2+ w_2(T^2-w_3^2)^2)+ (w_2^2-w_3^2)^2 ST = w_2 w_3(T^2-S^2)^2\\
    + T((w_2^2-S^2)^2 w_3 + (w_3^2-S^2)^2 w_2).
\end{gather*}

To solve these equations, we compute the resultant of the first and the $j$-th equation for $j\in \{2,3,4\}$ and enumerate the solutions. Due to the symmetries, we take $j=4$ as an example and eliminate $T$ (or $w_3$) from the two equations, which gives $w_1^2 (w_1+w_2+Z)\, \mathcal{G}=0$, where
\[
\mathcal{G}=Z^2(Z+ w_1 + w_2 )(w_1^2 + w_1 w_2 + w_2^2)+w_1 w_2(w_1 + w_2)( Z (w_1+ w_2) + w_1 w_2 ),\quad Z=T \ \mbox{or}\ w_3.
\]
It is easy to check that $w_1=0$ or $w_1+w_2+Z=0$ gives either $x_0=0$ or $\boldsymbol{a}$ is of type \eqref{equ-atype}. The latter case leads to $\frac{\partial \mathcal{F}_3}{\partial y_j}\equiv 0$ for some $j\in \{2,3,4\}$,  contradicting to the assumption in Lemma \ref{lem-F3}. For left possible solutions, $(w_3,T)$ should be two distinct real roots of $\mathcal{G}$ in $Z$ (the case $w_3=T$ implies above trivial solutions). However, a calculation on the discriminant of this equation gives that $\mathcal{G}$ never has all roots real.

As a consequence, we get $(\beta_2,p_2)=(-1,0)$ in condition (b) of Theorem \ref{thm-algo}. Finally, taking $\alpha=\gamma_1=(\frac{1}{3},\frac{1}{3},\frac{1}{3})$ in Theorem \ref{thm-algo}, we have
\[
M(\mathcal{F}_3) \curlyeqprec \mathrm{max}\left\{\left(-\frac{5}{6},0\right),\ (-1,1)\right\}= \left(-\frac{5}{6},0\right).
\]
The proof of Lemma \ref{lem-F3} is completed.
\end{proof}

Finally, in the rank 1 case we get $(\beta_1,p_1)=(-\frac{4}{3},0)$ in Theorem \ref{thm-algo}.

\textbf{(2) The rank 2 case.}
 Like the rank 1 case, by symmetry we may assume that at the critical point $\PP+g$ is the 9-parameter deformation,
\[
W_1 := W_1(z,\boldsymbol{a},\boldsymbol{b},\boldsymbol{c}) = c_0+c_1(z_1+a_2 z_2+a_3 z_3+a_4 z_4)^2+c_2(z_2+b_1 z_1+b_3 z_3+b_4
z_4)^2+\PP,
\]
where $\boldsymbol{a}=(a_2,a_3,a_4)$, $\boldsymbol{b}=(b_1,b_3,b_4)$, $c_1c_2\ne 0$ and $a_2 b_1\ne 1$ (since the rank is 2).

In the sequel, the strategy is the same as the proof of $M(\widetilde{\mathcal{F}_3})\curlyeqprec (-\frac{5}{6},0)$ in the proof of Lemma \ref{lem-F3} (under \eqref{equ-Q23}). We first introduce a little notation, let
\[
\mu_1:=a_2 b_1-1,\ \ (\mu_2,\mu_3,\mu_4,\mu_5):=\frac{1}{\mu_1}(a_3 b_1-b_3,\, a_4 b_1-b_4,\, a_2 b_3- a_3,\, a_2 b_4 - a_4).
\]

As before, in $W_1$ we set $y_1=z_1+a_2 z_2+a_3 z_3+a_4 z_4$, $y_2=z_2+b_1 z_1+b_3 z_3+b_4
z_4$ and $y_k=z_k$ for $k=3,4$.
By a direct computation parallel to that of $\widetilde{\mathcal{F}_3}$, it suffices to prove the following  counterpart of \eqref{equ-Q23},
\begin{equation}\label{equ-R34}
    M\left(W_3+ \widetilde{c_1}W_{4}^2+ \widetilde{c_2}W_{5}^2\right)\curlyeqprec \left(-\frac{1}{3},0\right),
\end{equation}
where $\widetilde{c_1}\widetilde{c_2}\ne 0$,
\begin{align*}
    W_3 = y_3 y_4 (y_3 + y_4)-& \Big((\mu_2 y_3 + \mu_3 y_4) + (\mu_4 y_3 + \mu_5 y_4)\Big)\\
    &\times\Big((1-\mu_2)y_3+ (1-\mu_3)y_4\Big)\Big((1-\mu_4)y_3+ (1-\mu_5)y_4\Big),
\end{align*}
and
\begin{align*}
    W_{4} =\ & (b_1-1)\Big((\mu_2 y_3+ \mu_3 y_4)+ (\mu_4 y_3+ \mu_5 y_4)-(y_3 + y_4)\Big)^2\\
    &+ (\mu_4 y_3 + \mu_5 y_4)^2 -b_1 (\mu_2 y_3 + \mu_3 y_4)^2,\\
    W_{5} = \ & (a_2-1)\Big((\mu_2 y_3+ \mu_3 y_4)+ (\mu_4 y_3+ \mu_5 y_4)-(y_3 + y_4)\Big)^2\\
    &+ (\mu_2 y_3 + \mu_3 y_4)^2 -a_2 (\mu_4 y_3 + \mu_5 y_4)^2.
\end{align*}

If $W_3\not\equiv 0$, then we finish by Van der Corput lemma. If $W_3\equiv 0$, using Lemma \ref{lem-coincide} we shall show that $\widetilde{c_1}W_{4}^2+ \widetilde{c_2}W_{5}^2$ do not possess a root with multiplicity four. Thus \eqref{equ-R34} is valid by Lemma \ref{lem-binary4}.

We classify all cases that $W_3\equiv 0$. A direct calculation shows that, except for the trivial solutions, it suffices to consider the following two cases:
\begin{align*}
    &(1)\quad \mu_3 =1,\quad \mu_2 = -\mu_4 = -\mu_5,\quad \mu_2(\mu_2^2 -(\mu_2+1))=0.\\
    &(2)\quad \mu_2 = \mu_5 =1, \quad \mu_3 = \mu_4, \quad \mu_3(\mu_3^2 -( \mu_3 +1))=0.
\end{align*}

We only consider (1), since (2) is similar. If $\mu_2\ne 0$, then $\mu_5^2 + \mu_5 =1$ and
    \[
    W_4 = (b_1-1)\mu_5\, y_3^2 + 2 y_3 y_4 - b_1 \mu_5 y_4^2,\quad W_5 = (a_2-1)\mu_5 \, y_3 ^2 - 2a_2 y_3 y_4 + \mu_5 y_4^2,
    \]
    with the discriminants  
    \[
    \Delta_{W_4} = 4\big(\mu_5^2\, b_1 (b_1-1)+1\big),\quad \Delta_{W_5} = 4\big(a_2^2 -\mu_5^2 a_2+ \mu_5^2\big).
    \]
    Regard them as the quadratic form in $b_1$ and $a_2$, respectively. Since $\mu_5^2 <4$, we know both $\Delta_{W_4}$ and $\Delta_{W_5}$ are positive. Therefore, noting that $a_2 b_1\ne 1$, a combination of Lemma \ref{lem-coincide} and Lemma \ref{lem-binary4} gives
   \eqref{equ-R34}.

   If $\mu_5=0$, then
   \[
   W_4 = (b_1-1)y_3^2 -b_1 y_4^2,\quad W_5 = (a_2-1)y_3^2 + y_4^2.
   \]
   We use Lemmas \ref{lem-binary4}, \ref{lem-coincide} again to get the conclusion.

   As a result, we obtain \eqref{equ-R34},  thus $M(\widetilde{\mathbf{P}}_4)\curlyeqprec (-\frac{4}{3},0)$ by Theorem \ref{thm-algo}.\\

\section{space-time estimates and nonlinear equations}\label{sec-nonlinear}

To begin with, for $1\leqslant q,r < \infty$, the mixed space-time Lebesgue spaces $L^q_t \ell^r$ are Banach spaces endowed with the norms
\begin{equation*}
    ||\widetilde{F}||_{L^q_t \ell^r} := \left( \int_{\R}\left(\sum_{x\in \Z^d}|\widetilde{F}(x,t)|^r\right)^{\frac{q}{r}}dt\right)^{\frac{1}{q}},
\end{equation*}
 with natural modifications for the case $q=\infty$ or $r=\infty$.
By an abuse of notations, for any function $\widetilde{F}=\widetilde{F}(x,t)$ defined on $\Z^d\times\R$, we write $\widetilde{F}=\widetilde{F}(t)$ sometimes for simplicity. Recall that $\sqrt{-\Delta}$ has been defined in \eqref{equ-solu-semigroup}.

In this section, the proofs follow the same line as those of \cite[Theorems 1.4, 5.9]{BCH23}. 

\begin{proof}[Proof of Theorem \ref{thm-stri}]
    By Duhamel's formula,
    \begin{equation}\label{equ-duhamel}
        u(x,t) = \cos (t\sqrt{-\Delta}) f_1(x) + \frac{\sin (t\sqrt{-\Delta})}{\sqrt{-\Delta}} f_2(x) + \int_0^t \frac{\sin(t-s)\sqrt{-\Delta}}{\sqrt{-\Delta}}F(x,s)ds.
    \end{equation}
    Considering the space $\ell^2(\Z^5)$, we set the operators $U_{\pm}(t):=\chi_{[0,\infty)}(t) \,e^{\pm it\sqrt{-\Delta}}$, where $\chi$ is the characteristic function. From Theorem \ref{thm-main thm}, we know that $|(G*f)(t)|_{\ell^{\infty}} \lesssim (1+|t|)^{-11/6}|f|_1$ provided $f\in \ell^1$. Then as a direct consequence of  Keel and Tao \cite[Theorem 1.2]{KT98}, we have
    \begin{equation*}
        \|e^{it\sqrt{-\Delta}} F_1\|_{L^{q}_t \ell^{r}} \lesssim |F_1|_2\quad\mbox{and}\quad 
        \left\|\int_{0}^t e^{i(t-s)\sqrt{-\Delta}}F_2(s) \, ds\right\|_{L^{q}_t \ell^{r}} \lesssim \|F_2\|_{L^{\Tilde{q}'}_t \ell^{\Tilde{r}'}}
    \end{equation*}
    for indices satisfying \eqref{equ-stri index} and $(F_1,F_2)\in \ell^2(\Z^d)\times L_t^{\tilde{q}'} \ell^{\tilde{r}'}$. This together with \eqref{equ-duhamel} gives
    \begin{equation*}
        \|u\|_{L^q_t \ell^r} \lesssim
        |f_1|_{2} + \left|\frac{1}{\sqrt{-\Delta}}f_2\right|_{2} + \left\|\frac{1}{\sqrt{-\Delta}}F\right\|_{L_t^{\tilde{q}'} \ell^{\tilde{r}'}}.
    \end{equation*}
    Finally, we utilize the boundedness of the operator $(\sqrt{-\Delta})^{-1}$, see e.g. \cite[Lemma 5.5]{BCH23}, to finish the proof.
\end{proof}

\begin{theorem}\label{thm-global-1}
    In \eqref{equ-wave equ}, let $F = |u|^{k-1}u$ with $k\ge 3$ and $f_1\equiv 0$. If $f_2\in \ell^1(\Z^5)$ with $|f_2|_1$ sufficiently small, then for any $p \in [2,+\infty]$, the global solution $u_{(k)}$ exists in $\ell^p$ with 
    \begin{equation*}
        |u_{(k)}(t)|_p \lesssim (1+|t|)^{-\frac{11}{6}\left(
        1-\frac{2}{p}\right)}.
    \end{equation*}
\end{theorem}

\begin{proof}
    For the convenience of notations we write $\zeta_p = \frac{11}{6}\left(
        1-\frac{2}{p}\right)$.
    We only prove the theorem for $p=k$, while the case for general $p \ge 2$ is similar. 
    
    For the contraction mapping principle to work, we need to estimate the $\ell^q$ norm of the solution, i.e. $G(t)*f_2$, in terms of time and the $\ell^p$ norm of initial data $f_2$. By interpolation between Theorem \ref{thm-main thm} and a trivial inequality $|G(t)| \leq C$, we deduce
    \begin{equation*}
        |G(t)|_k \leq C_k(1+|t|)^{-\zeta_k}.
    \end{equation*}
    Therefore, by Young's inequality it holds that
    \begin{equation}\label{equ-lplq}
        |G(t)*f_2|_q\leq |G(t)|_r |f_2|_p \leq C_r (1+|t|)^{-\zeta_r} |f_2|_p
    \end{equation}
    where $\frac{1}{r}=1+\frac{1}{q}-\frac{1}{p}$.

    Now we consider the metric space
    \begin{equation*}
        \mathcal{M}:=\left\{\widetilde{F} :\Z^5 \times \R \rightarrow \C , \|\widetilde{F}\|_{\mathcal{M}}= \sup_{t\in\R}(1+|t|)^{\zeta_k}|\widetilde{F}(\cdot,t)|_{k} \leq 2C_0|f_2|_{1}\right\},
    \end{equation*}
    with $C_0=C_0(k)$ to be determined later and the map $\Lambda$ on $\mathcal{M}$,
        \begin{equation*}
        \Lambda \widetilde{F} :=\Lambda \widetilde{F}(t) = G(t)*f_2 + \int_{0}^{t}G(t-s)*F(\widetilde{F}(s)) \, ds.
    \end{equation*}
    Given that $\widetilde{F} \in \mathcal{M}$, we see that
    \begin{equation*}
        (1+|t|)^{\zeta_k}
        |\Lambda \widetilde{F}(t)|_{k} \leq
        |f_2|_1+\int_0^t \left(
        \frac{1+|t|}{1+|t-s|} \right)^{\zeta_k} |\widetilde{F}(s)|_k^k \,ds
        \leq |f_2|_1 + \|\widetilde{F}\|_{\mathcal{M}}^k\mathbf{U}_k,
    \end{equation*}
    where we have used \eqref{equ-lplq} and $\mathbf{U}_k=\int_{\R}\, (1+|s|)^{(1 - k)\zeta_k}\, ds < \infty$, verified by $k \geq 3$. Taking the supremum, choosing proper $C_0$ and supposing $|f_2|_1$ sufficiently small in order, we know $\Lambda \widetilde{F} \in \mathcal{M}$. One can also check that 
    \begin{equation*}
        \|\Lambda u_1-\Lambda u_2\|_{\mathcal{M}}
        \leq C_k
        \mathbf{U}_k(2C_0|f_2|_1)
        ^{k-1}\|u_1-u_2\|_{\mathcal{M}}.
    \end{equation*}
    Therefore, $\Lambda$ is a contraction as long as $|f_2|_1$ is sufficiently small. Finally, $\Lambda$ admits a fixed point in $\mathcal{M}$, which is the global solution to \eqref{equ-wave equ}. 
\end{proof}

\section{Appendix}\label{Appendix}
Let $d=5$ and $\mathbf{P}_4$ be as in \eqref{equ-PP}, note that $\mathbf{P}_4$ is finitely determined (cf. \cite[Chapter 5]{M21}). We prove that the oscillation index (cf. Section \ref{ssec-newton}) of $\mathbf{P}_4$ at $0$ is $-\frac{11}{6}$ with multiplicity $0$. Then one can verify \eqref{equ-sharp} by the proof of \cite[Lemma 3.6]{BCH23} and the following proof with slight modifications. 

Recall that $\PP(\xi)=(\sum_{j=1}^4 \xi_j)^3-\sum_{j=1}^4 \xi_j^3$ for $\xi\in \R^4$, the Newton distance $d_{\PP}=\frac{3}{4}$. By the additivity of the oscillation index, it suffices to prove the following assertion.
\begin{prop}\label{prop-PP4}
    There exist $\psi\in C_0^\infty(\R^4)$ with $\psi(0)\ne 0$, and $c_\psi \ne 0$ such that
    \[
    \int_{\R^4} e^{i \lambda \PP(\xi)} \psi(\xi)\, d\xi = c_\psi\, \lambda^{-\frac{4}{3}} + \mathcal{O}\left(\lambda^{-\frac{3}{2}}\log \lambda\right),\quad\mbox{as}\ \ \lambda\rightarrow +\infty.
    \]
\end{prop}
\begin{proof}
    Using a change of variables $$\xi_1+\xi_2=2 w_1,\ \xi_1-\xi_2=2 w_2,\ \xi_3+\xi_4=2 w_3,\ \xi_3-\xi_4=2 w_4,$$ it suffices to consider $J(\lambda, f,\psi)$ for proper $\psi$ (see \eqref{equ-JJJ2} below), where
    \[
    f(w) = 4(w_1+w_3)^3 -(w_1^3+w_3^3) - 3 w_1 w_2^2 -3 w_3 w_4^2,\quad w\in\R^4.
    \]
    Let $C\gg 1$ be a constant, and
    \begin{align*}
        & D_1 = \left\{w: |w_1|\le C\lambda^{-1},|w_3|\le C\lambda^{-1}\right\},  \ \ 
         D_2 = \left\{w: |w_1|\ge C\lambda^{-1},|w_3|\ge C\lambda^{-1}\right\},\\
        & D_3 = \left\{w: |w_1|> C\lambda^{-1},|w_3|< C\lambda^{-1}\right\}, \ \  D_4 = \left\{w: |w_1|< C\lambda^{-1},|w_3|> C\lambda^{-1}\right\}.
    \end{align*}
We write
\[
J(\lambda,f,\psi) = \int_{D_1}+ \int_{D_2}+\int_{D_3}+ \int_{D_4}=: J_1+ J_2+ J_3+ J_4.
\]
    
    It is easy to see that $J_1 = \mathcal{O}(\lambda^{-2})$, while on $D_3$, the method of stationary phase gives
    \begin{equation}\label{equ-sta}
         \int_{-\infty}^{+\infty} e^{i\lambda w_1 w_2^2}\,\psi(w)\,dw_2 =  \frac{c}{\sqrt{\lambda|w_1|}}\psi(w_1,0,w_3,w_4) + \mathcal{O}(\lambda^{-1}|w_1|^{-1}),\quad \mbox{as} \ \ \lambda\rightarrow+\infty,
    \end{equation}
  Inserting this to $J_3$ yields an upper bound $|J_3|\lesssim\lambda^{-\frac{3}{2}}$. The same estimate is valid for $J_4$.
Now it suffices to consider $J_2$. Similar to \eqref{equ-sta}, we obtain that
    \begin{equation}\label{equ-JJJ2}
        J_2 = c\lambda^{-1}\int_{\mathcal{V}}\, e^{i\lambda (4(w_1+w_3)^3-(w_1^3+w_3^3))}|w_1 w_3|^{-\frac{1}{2}}\psi_1(w_1,w_2)\, dw_1 dw_2  + \mathcal{O}(\lambda^{-2}\log^2\lambda),
    \end{equation}
    provided $\lambda$ large enough, where $\mathcal{V}=\{(w_1,w_3)\in\R^2:|w_1|>C\lambda^{-1},|w_3|>C\lambda^{-1}\}$, and $\psi_1=\psi_0((w_1+w_3)^6,w_1^6 w_3^6)$ for some $\psi_0\in C^\infty(\R^2)$ supported in $B_{\R^2}(0,1)$. 
    
     Denote by $\mathbf{K}$ the integrand in \eqref{equ-JJJ2}, by symmetries it suffices to estimate
    \[
    \mathcal{K}_1 = \int_{\mathcal{V}\,\cap\, \mathcal{D}}\, \mathbf{K}\,dw_1 dw_2 \quad \mbox{and}\  \quad \mathcal{K}_2 =\int_{\mathcal{V}\,\cap\, \{(w_1,w_3)\in\R^2:w_1>0,\, w_3<0\}}\,\mathbf{K}\,dw_1 dw_2,
    \]
where $\mathcal{D}$ is the triangle
with vertices $(C\lambda^{-1},C\lambda^{-1}),$ $(1/2,1/2)$ and $(1-C\lambda^{-1},C\lambda^{-1})$.
    We begin with $\mathcal{K}_1$, a change of variables $ s_1=w_1+w_3$, $s_2=w_1 w_3$ gives
    \begin{equation}\label{equ-J5}
        \mathcal{K}_1 = \int_{\frac{2}{\lambda}}^1  \int_{\frac{1}{\lambda^2}}^{\frac{s_1^2}{4}} \frac{e^{3 i \lambda s_1(s_1^2+s_2)}}{\sqrt{s_2(s_1^2-4s_2)}}\,\psi_1(s_1,s_2)\, ds_2 ds_1+ \mathcal{O}(\lambda^{-\frac{1}{2}}\log \lambda),\quad\lambda\rightarrow+\infty,
    \end{equation}
    where the remainder comes from the integral on the triangle with vertices $(2C\lambda^{-1}, C^2 \lambda^{-2}),$ $(1,\lambda^{-1}(1-\lambda^{-1}))$ and $(1,1/4)$ in $(s_1,s_2)$-plane.

    A change of variable $r= 4s_2/s_1^2$ gives
    \begin{align*}
        \mathcal{K}_1 &= \frac{1}{2}\int_{\frac{4}{\lambda^2}}^1 \left(\int_{\frac{2}{\lambda \sqrt{r}}}^1 e^{3 i \lambda s_1^3 (1+r/4)}\, \psi_1(s_1,s_1^2 r/4)\,ds_1\right)\frac{dr}{\sqrt{r(1-r)}}+ \mathcal{O}(\lambda^{-\frac{1}{2}}\log \lambda)\\
        & = \frac{1}{2}\int_{\frac{4}{\lambda^2}}^1 \left(\int_{0}^1\, e^{3 i \lambda s_1^3 (1+r/4)}\, \psi_1(s_1,s_1^2 r/4)\,ds_1\right)\frac{dr}{\sqrt{r(1-r)}}+ \mathcal{O}(\lambda^{-\frac{1}{2}}\log \lambda).
    \end{align*}
Using the method of stationary phase (cf. \cite[Chapter 8\, \S 5]{S93}) to the inner integral yields the asymptotic $\mathcal{K}_1 = c_1 \lambda^{-\frac{1}{3}}+ \mathcal{O}(\lambda^{-\frac{1}{2}}\log \lambda)$ as $\lambda\rightarrow+\infty$.

Similar argument can be applied to $\mathcal{K}_2$. Indeed, a simple computation gives that
\[
\mathcal{K}_2  = \int_{\{(w_1,w_3)\in\R^2:\,0<w_1<1,\,-1<w_3<0\}}\,\mathbf{K}\,dw_1 dw_2  + \,\mathcal{O}(\lambda^{-\frac{1}{2}}).
\]
In coordinate systems $(s_1,s_2)$ where $s_1=w_1+w_3$ and $s_2=-w_1 w_3$, we have
\begin{align*}
    \mathcal{K}_2 = \int_0^{+\infty} \left( \int_{-\infty}^{+\infty} \frac{e^{i\lambda s_1 (s_1^2-s_2)}}{\sqrt{s_1^2+ 4s_2}}\,\psi_1(s_1,s_2)\, ds_1\right) \frac{ds_2}{\sqrt{s_2}} + \,\mathcal{O}(\lambda^{-\frac{1}{2}})=: \mathcal{K}_3+\,\mathcal{O}(\lambda^{-\frac{1}{2}}).
\end{align*}
Writing $r=s_1/\sqrt{s_2}$ gives
\begin{align*}
 \mathcal{K}_3 &= \int_0^{+\infty} \left(\int_{-\infty}^{+\infty} \frac{e^{i \lambda s_2^{\frac{3}{2}}r(r^2-1)}}{\sqrt{r^2+4}} \psi_1(\sqrt{s_2}\,r,s_2)\, dr \right) \frac{ds_2}{\sqrt{s_2}}\quad (\mbox{let}\ \tau=s_2^{\frac{3}{2}}) \\
&= \int_{-\infty}^{+\infty} \left(\int_0^{+\infty} e^{i\lambda \tau\, r(r^2-1)}\, \tau^{-\frac{2}{3}}\, \psi_1(\tau^{\frac{1}{3}}r,\tau^{\frac{2}{3}})\,d\tau\right) \frac{dr}{\sqrt{r^2+4}}.
\end{align*}

Noting that $r(r^2-1)$ has zeros $0,\pm 1$, we split 
$$
\mathcal{K}_3=\int_{|r|<C\lambda^{-1}} + \int_{C\lambda^{-1}<|r|<1-C\lambda^{-1}} + \int_{|r-1|<C\lambda^{-1}}+ \int_{|r|>1+C\lambda^{-1}}=:+\mathcal{K}_4+\mathcal{K}_5+\mathcal{K}_6+\mathcal{K}_7.
$$
Obviously, $\mathcal{K}_4=\mathcal{O}(\lambda^{-1})$ and $\mathcal{K}_6=\mathcal{O}(\lambda^{-1})$. For the left two terms we use stationary phase method as in the estimate of $\mathcal{K}_1$ (for $\mathcal{K}_7$ we use a change of variable $y= r^3 \tau$ in addition), and get that both of them have asymptotics $\lambda^{-\frac{1}{3}}$ as $\lambda\rightarrow \infty$.

As a consequence, we check that there exists $c_0\ne 0$ such that $J_2 =c_0 \lambda^{-\frac{4}{3}}+ \mathcal{O}(\lambda^{-\frac{1}{2}}\log\lambda)$, which completes the proof of Proposition \ref{prop-PP4}.
\end{proof}

\section*{Acknowledgement}

C.Bi is grateful to Prof. James Montaldi and Titus Piezas III for helpful discussions and suggestions. B.Hua is supported by NSFC, No. 12371056, and by Shanghai Science and Technology Program [Project No. 22JC1400100].

\printbibliography

\end{document}